\documentclass[a4paper]{amsart}
\usepackage[english]{babel}
\usepackage{amssymb}
\usepackage{amsmath}
\usepackage{amsthm}
\usepackage{graphicx}
\usepackage{fancyhdr}
\usepackage{bbm}
\usepackage{multirow}
\newfont{\bssten}{cmssbx10}
\newfont{\bssnine}{cmssbx10 scaled 900}
\newfont{\bssdoz}{cmssbx10 scaled 1200}

\usepackage{latexsym,amssymb,amsthm,amsxtra}
\usepackage{amsmath}
\usepackage{stmaryrd}
\usepackage{mathrsfs}
\usepackage{tikz}
\usepackage{pgf}
{\bf}{\it}
{\bf}{\it}
\newtheorem{theorem}{Theorem}
\newtheorem{definition}{Definition}
\newtheorem{lemma}{Lemma}

\newtheorem{proposition}{Proposition}
\newtheorem{corollary}{Corollary}

\def\/{\, | \,}

\def\N{{\mathbb N}}

\def\I{{\mathbb I}}

\newcommand{\pae}[1]{\mbox{$\lfloor \kern-1pt #1 \kern-1pt \rfloor$}}
\newcommand{\paep}[1]{\mbox{$\lceil \kern-1pt #1 \kern-1pt \rceil$}}

\def\E{{\mathcal E}}

\def\N{{\mathbb N}}
\def\R{{\mathbb R}}
\def\maV{{\mathcal V}}

\def\maN{{\mathcal N}}
\def\maU{{\mathcal U}}
\def\maE{{\mathcal E}}

\def\v{{\--}}
\def\pv{{\not\!\!\--}}

\newcommand\suite[1]{\left\{#1\right\}_{n\in\N}}
\newcommand\suitek[1]{\left\{#1\right\}_{k\in\N}}

\newcommand\maM{{\mathcal M}}

\newcommand\maI{{\mathcal I}}
\newcommand\maS{{\mathcal S}}

\def\isdef{\triangleq}

\def\E{\mathcal E}

\def\N{\mathbb N}

\def\isdef{\triangleq}

\def\pv{{\not\!\!\--}}

\usepackage[utf8]{inputenc}

\begin{document}
	
	\title[Stability regions]{Stability regions of systems with compatibilities, and ubiquitous measures on graphs}

	\author[Begeot, Marcovici, Moyal]{Jocelyn Begeot, Irène Marcovici and Pascal Moyal}
	
	\begin{abstract} 
	This paper addresses the ubiquity of remarkable measures on graphs, and their applications. 
	In many queueing systems, it is necessary to take into account the compatibility constraints between users, or between supply 
	and demands, and so on. The stability region of such systems can then be seen as a set of measures on graphs, where the measures under consideration represent the arrival flows to the various classes of users, supply, demands, etc., and the graph represents the compatibilities between those classes. 
	In this paper, we show that these `stabilizing' measures can always be easily constructed as a simple function of a family of weights on the edges of the graph. Second, we show that the latter measures always coincide with invariant measures of random walks on the graph under consideration. 
	\end{abstract}

\maketitle
\section{Introduction}
\label{introduction}
Queueing models with compatibilities have recently gained an increasing interest, both from a theoretical and from an applied standpoint. 
In the original skill-based routing problems, customers and servers are bound by compatibilities constraints, and it is the role of the routing algorithm to determine which customer should be matched to which server upon arrival, see e.g. \cite{TE93}, and the recent survey \cite{CDS20}. 
However, in many applications customers and servers may play symmetric role, such as in peer-to-peer applications, jobs search, dating websites, housing programs, organ transplant allocations, supply and demands, and so on. In all such applications, the system is basically just an interface used by the elements, to be matched, and to depart the system by pairs. 

To account for this variety of applications, in \cite{CKW09} (see also \cite{AdWe}), a variant of skill-based systems was introduced, commonly referred to as `Bipartite matching models' (BM): couples customer/server enter the system at each time point, and servers come and go into the system exactly as customers. A (bipartite) graph, representing the class of users, supplies, demands, etc., determines the compatibility between customers and servers. Each time a couple customer/server is formed, the latter leaves the system right away. The matching policy addressed in these seminal works is `First Come, First Served' (FCFS). 
Interestingly, the stationary state in such cases is shown to enjoy a remarkable product form (see \cite{ABMW17}), that can be generalized to a wider class of systems with compatibilities, such as matching queues implementing the so-called FCFS-ALIS (Assign the Longest Idling Server) service discipline - see e.g. \cite{AW14}, or \cite{AKRW18} for an overview of various extensions of BM models or redundancy models. Let us also mention other approaches to such bipartite models, such as fluid and diffusion limits for systems with compatibility probabilities \cite{BC17}, applications to taxi hubs \cite{BC15} and ride sharing  \cite{OW20}. 
The original settings of \cite{CKW09, AdWe} have then been generalized to the case of non independent arrivals and other matching policies ({\em extended} bipartite matching model - see \cite{BGM13,MBM18}). 

To take into account applications such as assemble-to-order systems, dating website, car sharing and cross-kidney transplants, 
a variant of this model has been more recently introduced, in which arrivals are made one by one, and the compatibility graph is general, i.e. not necessarily bipartite: this is the so-called General stochastic matching model (GM), introduced in \cite{MaiMoy17}. \cite{MBM21} shows that the GM model under the First Come, First Matched matching policy also enjoys a product form in steady form. Then, GM was studied along various angles, among which, fluid limits of a continuous-time variant \cite{MoyPer17}, optimization \cite{NS17} or optimal control \cite{GW14,CBD19}, matching models on hypergraphs \cite{RM21}, on graphs with self-loops \cite{BMMR21}, or models with reneging, see \cite{BDPS11,JMRS20,ADWW21}. Recently, GM models have been shown to share remarkable similarities with order-independent loss queues, see \cite{Com21}.

For all these models, the stability regions are expressed as a set of measures on nodes representing the arrival rates to each class of users, such that the system is positive recurrent. (These measures are probability measures for discrete-time models, and arrival intensities for continuous-time models.) As will be specified below, the conditions on these measures typically take the form of a constraint on the arrival rates to all subsets of nodes, to be less than the arrival rates to the subset of the neighboring nodes - a condition that is reminiscent of the Hall condition for the existence of a perfect matching on a graph, see \cite{Hall35}. In this work, our aim is to characterize exactly these sets of measures. While in such systems, the typical procedure is to construct an optimal control (a Markovian matching policy, for instance) that is able to optimize a given criterion {\em given the arrival rates} (that are thus seen as a 
constraint to the problem), hereafter we somehow reverse this procedure: having fixed a control that is able to achieve stability (the `First Come First Matched' policy for general matching models, the `Match the Longest' policy for bipartite models), we make explicit, the construction procedure of a set of arrival rates rendering the system stable. This aim is natural for practical purposes: being able to explicitly construct a set of arrival rates rendering the considered system stable, provides the feasible admission controls for a given graph topology, and a given Markovian control of the system. 

We obtain two remarkable results: (i) we show that in many cases, the `stabilizing' arrival rates of the considered system can {\em always} be constructed as a simple weighted measure on the edges of the graphs and (ii) we deduce from this, that a stabilizing set of arrival rates always coincide with an invariant measure of a random walk on the graph. As a by-product of the first point, for some graph topologies we are able to 
determine {uniquely} the matching rates of the various edges, {\em independently} on the matching policy. 

Beyond its practical interest for admission control, this ubiquity of the stabilizing measures of matching models draws an insightful correspondance between stochastic matching models and random walks on graphs which, we believe, opens the way for an interesting avenue of research. 

This paper is organized as follows. After some preliminaries in Section \ref{sec:prelim}, in Section \ref{sec:results} we present our main results, establishing various representations of the stabilizing measures of matching models in the cases of non-bipartite multigraphs, of bipartite graphs, and of bipartite skill-based queueing systems, respectively. Proofs for our three main results are provided respectively in Sections \ref{sec:proofmain}, \ref{sec:proofmainBip} and \ref{sec:proofqueueing}. Last, in Section \ref{sec:appli} we show how our results can be applied, to show the insensitivity of the matching rates to the matching policy of the considered matching model.  

\section{Preliminaries}
\label{sec:prelim}

\subsection{General notations}

We denote by $\R$ the set of real numbers, by $\mathbb{R}_+$ the set of non-negative real numbers and by $\mathbb{R}_+^*$ the subset of positive real numbers. Likewise, we denote by $\N$ the set of non-negative integers and by $\N_*$ the subset of positive integers. For $a$ and $b$ in $\mathbb{N}$, we denote by $\llbracket a;b \rrbracket$ the set $[a;b]\cap\mathbb{N}$. 
For any finite set $A$, we denote by $|A|$ the cardinality of $A$ and by $\mathcal{P}(A)$ the set of all the subsets of $A$.
For any set $E$, we denote by $\mathbf 1_E$ the indicator function of $E$. 

The set of positive measures having full support on $A$ is denoted by $\mathcal{M}(A)$, whereas the set of probability measures having full support on $A$ is denoted by $\overline{\mathcal{M}}(A)$. For any $\mu\in\maM(A)$, we denote by $\bar\mu\in\overline{\maM}(A)$, its normalized probability measure, namely $\bar\mu(.)={\mu(.)\over \mu(A)}$. Let us also denote by ${\mathcal{M}^{\ge 0}}(\mathcal{A})$, the set of positive measures on $A$ having support included in $A$. 

\subsection{Multigraphs}

Hereafter, a {\em multigraph} is given by a couple $G=(\maV,\maE)$, where $\maV$ is the (finite) set of nodes and $\maE\subset \maV\times \maV$ is the set of edges. All multigraphs considered hereafter are undirected, that is, $(u,v)\in \maE$ $\implies$ $(v,u)\in \maE$, for all $u,v \in \maV$.  We write $u \v v$ or $v \v u$ for $(u,v) \in \maE$, 
and $u \pv v$ (or $v \pv u$) else. Elements of the form $(v,v) \in \maE$, for $v\in\maV$, are called {\em self-loops}. A multigraph having no self-loop is simply a graph. With respect to the classical notion of multigraphs, we assume hereafter that all the edges are simple. 

For any multigraph $G=(\maV,\maE)$ and any $U\subset\maV$, we denote 
\[
\maE(U) \isdef \{ v \in \maV  :  \exists u \in U, \ u
\-- v\}
\]
the neighborhood of $U$, and for $u\in \maV$, we write for short $\maE(u)\isdef\maE(\{u\})$. 

An {\em independent set} of $G$ is a non-empty subset $\maI \subset \maV$ which does not include any pair of neighbors, i.e.: $\forall (i,j) \in \maI^2, \; i \pv j$. 
We let $\I(G)$ be the set of independent sets of $G$. A graph $G=(\mathcal{V},\mathcal{E})$ is said \textit{bipartite} if the set of its vertex $\mathcal{V}$ can be partitioned into two independent sets 
$\mathcal{V}_1$ and $\mathcal{V}_2$, namely, each edge connects an element of $\mathcal{V}_1$ to an element of $\mathcal{V}_2$. 

Throughout this paper, all considered (multi)graphs are connected, that is, for any $u,v \in \maV$, there exists a subset 
$\{v_0\isdef u, v_1, v_2,\dots, v_p\isdef v\}\subset \maV$ such that $v_i \v v_{i+1}$, for any $i \in\llbracket 0;p-1 \rrbracket$. Moreover, we consider that they have at least two nodes, i.e. $|\mathcal{V}|\geq 2$.

\subsection{Weighted measures on a multigraph}

For any multigraph $G=(\mathcal{V},\mathcal{E})$, we say that a family $\alpha=(\alpha_{i,j})_{(i,j)\in\maE}$ of real numbers is a family of \emph{weights} on $\mathcal{E}$ if for all $(i,j)\in\maE, \; \alpha_{i,j}=\alpha_{j,i}>0$.
Hereafter, a family of weights $\alpha$ on $\mathcal{E}$ will be seen as an element of the set $\mathcal{M}(\mathcal{E})$, and for any such $\alpha$ and any edge $e=(i,j)\in\mathcal{E}$, we will write indifferently $\alpha_e$, $\alpha_{i,j}$ or $\alpha_{j,i}$. 

For any family of weights $\alpha\in{\mathcal{M}}(\maE)$,  
we define the associated positive measure (on nodes) $\mu^\alpha\in \mathcal{M}(\mathcal V)$, by 
\begin{equation}
\label{eq:deflambdaalpha}
\mu^\alpha (i) \isdef \sum\limits_{j\in\mathcal{E}(i)}\alpha_{i,j},\quad i\in\mathcal{V},
\end{equation}
and the associated probability distribution 
$\bar\mu^\alpha:=\overline{\mu^\alpha}\in\overline{\mathcal{M}}(\mathcal{V})$, by 
\begin{equation}
\label{eq:defmualpha}
\bar\mu^\alpha (i) = \frac{\sum\limits_{j\in\mathcal{E}(i)}\alpha_{i,j}}{\sum\limits_{l\in\mathcal{V}}{\sum\limits_{j\in\mathcal{E}(l)}{\alpha}_{l,j}}},\quad i\in\mathcal{V}.
\end{equation}
The (probability) measures $\mu^\alpha$ and $\bar\mu^\alpha$, $\alpha\in{\mathcal{M}}(\maE)$, will be called hereafter  \textit{weighted measures.} 

\newpage
\subsection{Stochastic processes of interest on a multigraph}

\subsubsection{Random walks.}

We say that a stochastic matrix $P=(P(i,j))_{(i,j)\in\maV^2}$ defines a \emph{random walk} on the edges of the multigraph $G=(\maV,\maE)$ if $P(i,j)=0$ as soon as $(i,j)\not\in\maE$. We say furthermore that the random walk is \emph{reversible} if the associated Markov chain admits a reversible invariant measure.

Given a family of weights $(\alpha_{i,j})_{(i,j)\in\maE}$, we introduce the corresponding \emph{weighted random walk}, defined by the stochastic matrix $P^{\alpha}$, with
\[P^\alpha(i,j) = {\alpha_{i,j}  \over \sum_{l\in\maE(i)} \alpha_{i,l}},\quad i\in\maV,\,\quad j\in \maE(i).\]
It is then immediate that $\mu^\alpha$ is a stationary measure for the associated Markov chain. Consequently, as the Markov chain is clearly irreducible on $\maV$, $\bar\mu\alpha$ is the unique stationary probability. 

\subsubsection{General stochastic matching model}
\label{subsec:Stability regions of the general stochastic matching model}

We consider the matching model introduced in \cite{MaiMoy17} (and generalized in \cite{BMMR21} for multi-graphs) for which we 
recall the main definitions, assumptions and stability results. 
For more details, see \cite{MaiMoy17} (for models on general graphs without self-loop), and \cite{BMMR21} for the same model with 
self-loops.  Consider a connected multigraph $G=(\maV,\maE)$. Nodes of $G$, i.e. elements of $\maV$, represent classes of items. 
We say that two items of respective classes $i$ and $j\in\maV$ are {\em compatible} if $(i,j)\in \maE$, that is, $i$ and $j$ are neighbors 
in $G$. Items enter the system one by one in discrete time, and, for all $n\in\N_*$, we denote by $V_n\in \maV$ the class of the item entering 
the system at time $n$. The sequence $\left\{V_n\right\}_{n\in\mathbb{N}_*}$ is supposed to be IID, of common distribution $\bar\mu \in \overline{\maM}(\maV)$. 

Then, upon the arrival of an element, say, of class $i$, we investigate whether there is in line a compatible element 
with the incoming item. If this is the case, then the incoming $i$-item is matched with one of these elements (if there are more than one), following a fixed criterion called {\em matching policy}, and the two matched elements 
leave the system right away. 
The matching policy is denoted by $\Phi$, and 
depends only on the arrival dates of the stored items, their classes, and possibly of a random draw that is independent of everything else 
- we then say that the policy is {\em admissible}. 
For instance, for the \emph{First Come, First Matched} policy $\Phi=\textsc{fcfm}$, the oldest compatible item in line is chosen as the match of $i$. 
See \cite{MaiMoy17,BMMR21} for other examples of admissible matching policies. The system is then fully characterized by the triple 
$(G,\Phi,\bar\mu)$. 

For any fixed connected multigraph $G$, any fixed admissible matching policy $\Phi$ and any fixed probability measure $\bar\mu$, 
the system can be represented by the Markov chain 
$\suite{W_n^{\Phi,\mu}}$, valued in a space denoted by $\mathbb W$.
For any connected multigraph $G=(\mathcal{V},\mathcal{E})$ and any admissible matching policy $\Phi$, we define the {\em stability region} associated to $G$ and $\Phi$ as the set of measures 
\begin{equation*}
\textsc{stab}(G,\Phi) \isdef \left\{\bar\mu \in \overline{\maM}\left(\maV\right)\,: \,\suite{W_n^{\Phi,\bar\mu}} \mbox{ is positive recurrent}\right\}.
\end{equation*}

\subsubsection{Continuous-time matching model} A continuous-time matching model associated to the triple $(G,\Phi,\mu)$, for $\mu\in\maM(\maV)$, is defined exactly as its diecrete-time analog, except that arrivals to each node $i$ are given by a Poisson process of intensity $\mu(i)$, independently of everything else. 

\subsubsection{Extended bipartite stochastic matching model}
\label{subsec:presentation of the bipartite matching model on discrete time}
In this section, we briefly recall the main definitions and results regarding the so-called extended bipartite matching model, as introduced in \cite{BGM13}. This model is a variant of the general stochastic matching model introduced in Section 
\ref{subsec:Stability regions of the general stochastic matching model}. For more details, the reader is referred to \cite{BGM13}. 

We are given a connected bipartite graph $G=(\maV_1\cup\maV_2,\maE)$. 
As above, elements of $\maV_1$ and $\maV_2$ represent classes of items, and we say that two items $i\in\maV_1$ and $j\in\maV_2$ are compatible if there is an edge between $i$ and $j$. 
In an extended bipartite matching model, exactly one element of class in $\maV_1$, and one element of class in $\maV_2$, 
 enter the system at each time point, in other words the arrivals are done {\em two-by-two}. 
 The sequence $\left\{V_n\right\}_{n\in\mathbb{N}_*}$ represents the couples of classes of incoming items; namely, for any $n\in \N_*$, 
$V_n=(i,j)$ for $i\in\maV_1$ and $j\in\maV_2$, means that an element of class $i$ and an element of class $j$ enter the system together at time $n$. The family $\left\{V_n\right\}_{n\in\mathbb{N}_*}$ is supposed to be IID, of generic distribution $\tilde\mu$ on $\maV_1\times\maV_2$ with marginals $\tilde\mu_1$ and $\tilde\mu_2$ respectively. 
We write $\tilde\mu=(\tilde\mu_1,\tilde\mu_2)$ a such distribution. Then, upon the arrival of each couple, say of a couple $(i,j)$, we first investigate whether there is in line an element of class in $\maV_1$ that is compatible with the incoming $j$-item. If this is the case, then the incoming $j$-item is matched with one of these elements (if there are more than one), that is chosen following a given matching policy 
$\Phi$, and the two elements leave the system right away. 
Then we apply the exact same procedure to the incoming $i$-item. If one of the two incoming item did not find a match while the other did, then the element is stored in the buffer. If none of the two did find a match in the buffer, then, either the two incoming items are compatible and then they are matched and leave the system right away, or they are not and they are both stored in the buffer.  
Again, the matching policy $\Phi$ is supposed to be admissible, i.e. to depend only on the arrival dates of the stored items, their classes, and possibly of a random draw independent of everything else. 
For instance, for $\Phi=\textsc{fcfm}$ ('First Come, First Matched'), the oldest compatible item in line is chosen, or, for $\Phi=\textsc{ml}$ ('Match the Longest'), an item 
of the compatible class having the largest number of elements in line is chosen, ties being broken uniformly at random. 
See again \cite{BGM13} for other examples. Altogether, the system is again fully characterized by the triple $(G,\Phi,\tilde\mu)$. 

For a fixed connected bipartite graph $G$, a fixed admissible matching policy $\Phi$ and a fixed probability measure $\tilde\mu$, the bipartite matching model associated to $(G,\Phi,\tilde\mu)$ can be fully represented by a Markov chain $\suite{Y_n^{\Phi,\tilde\mu}}$, that is valued in a state space that we denote by $\mathbb{Y}$.
For any bipartite $G=(\maV_1\cup\maV_2,\maE)$ and any admissible $\Phi$, we define the {\em stability region} associated to $G$ 
and $\Phi$ as the set of measures 
\begin{equation*}
\textsc{stab}_{\textsc{b}}(G,\Phi) \isdef \Bigl\{\tilde\mu \in \overline{\mathcal{M}}\left(\mathcal{V}_1\times\mathcal{V}_2\right): \suite{Y_n^{\Phi,\tilde\mu}} \mbox{ is positive recurrent}\Bigl\}.
\end{equation*}


\section{Results}\label{sec:results}

\subsection{General case}
Our main theorem below presents several equivalences. Most of them were already known, but in order to emphasize on the ubiquitous nature of the measures described in this result, we gather all the results in a single theorem.

\begin{theorem}\label{thm:main}
Let $G=(\maV,\maE)$ be a connected multigraph that is not a bipartite graph, and let $\mu\in\mathcal M(\maV)$. The following properties are equivalent.
\begin{enumerate}
\item[{\it (i)}] The measure $\mu$ is invariant for a reversible random walk on the edges of the multigraph $G$.
\item[{\it (ii)}] The measure $\mu$ is a weighted measure, that is, there exist weights $(\alpha_{i,j})_{(i,j)\in\maE}$ such that for all $i\in\maV$,
$\mu(i) = \mu^\alpha(i)$.  
\item[{\it (iii)}] The measure $\mu$ satisfies: for any non-empty subset $U$ of $\maV$ with $U\not=\maV$, 
$\mu\left(\maU\right) < \mu\left(\maE\left(\maU\right)\right).$
\item[{\it (iv)}] The measure $\mu$ satisfies: for any independent set $\mathcal I$ of $\maV$, 
$\mu\left(\mathcal I\right) < \mu\left(\maE\left(\mathcal I\right)\right).$
\item[{\it (v)}] The measure $\bar\mu$ belongs to $\textsc{stab}(G,\textsc{fcfm})$, that is: 
if items arrive according to the distribution $\bar\mu$, then the general stochastic matching model $(G,\textsc{fcfm},\bar\mu,)$ is stable. 
\item[{\it (vi)}] The continuous time-general stochastic matching model $(G,\textsc{fcfm},\mu)$ 
is stable. 
\end{enumerate}
\end{theorem}
Theorem \ref{thm:main} is proven in Section \ref{sec:proofmain}. 
For any given connected multigraph $G=(\mathcal{V},\mathcal{E})$, defining the sets 
\begin{align*}
\mathcal W(G) &\isdef\left\{ \mu^\alpha : \alpha\in\mathcal{M}(\maE)\right\}\subset \maM(\maV),\\
\mathcal{N}(G) &\isdef \left\{\mu \in \maM\left(\maV\right):\forall \,\mathcal U\in\mathcal P(\maV)\setminus\{\emptyset,\maV\},\; \mu\left(\maU\right) < \mu\left(\maE\left(\maU\right)\right)\right\}.
\end{align*}
the equivalence $(ii)\iff (iii)$ can be reformulated as 

\begin{equation}
\label{eq:setequality}
\mathcal N(G) = \mathcal W(G),\mbox{ for any $G$ that is not a bipartite graph.}
\end{equation}
This result has an immediate practical interest for admission control, in the 
wide class of systems described above: it makes precise the exact set of 
arrival rates (or probabilities) for which the corresponding system can be stabilizable, and provides a simple way 
of constructing such arrival rates, and so to calibrate the admission control of the corresponding arrivals into the system. For this, it is necessary and sufficient to set any family $\alpha$ of weights on the edges of $G$, and then to set $\mu\in\mathcal N(G)$ (resp., $\bar\mu\in\overline{\mathcal N}(G)$) in function of $\alpha$ according to (\ref{eq:deflambdaalpha}) (resp., (\ref{eq:defmualpha})). 
Moreover, an explicit representation of the family of weights 
$\alpha$ such that $\mu=\mu^\alpha$, is provided in the proof of Theorem \ref{thm:main}, 
as a simple function of the matching rates of a related stochastic matching model, 
see Corollary \ref{cor:stabpoids} and (\ref{eq:defTheta}). 

\subsection{Bipartite case} 
Clearly, ${\maN}(G)$ is empty whenever $G$ is bipartite, of bipartition, say, 
$\maV=\maV_1 \cup \maV_2$. Indeed, for any $\mu$ in the latter, we we would have 
$\mu(\maV_1)=\mu\left(\maE(\maV_2)\right) > \mu\left(\maV_2\right)$, 
and symmetrically $\mu(\maV_2) > \mu(\maV_1)$. Specifically, from Proposition 4.2 in \cite{BMMR21} and 
the equivalence $(iii) \iff (iv)$, we obtain that the set $\mathcal N(G)$ is non-empty if and only if $G$ is not a bipartite graph. 

In fact, as is observed in the Introduction, the case where $G=(\mathcal{V},\mathcal{E})$ is a bipartite graph is relevant in various applications. 
Hereafter, for a bipartite graph $G=(\mathcal{V}_1\cup\maV_2,\mathcal{E})$ and a measure $\mu\in\maM(\maV)$, let us define the measures 
$\tilde\mu_1\in\overline\maM(\maV_1)$ and $\tilde\mu_2\in\overline\maM(\maV_2)$, by 
\begin{equation}
\label{eq:deftildemu}
\tilde\mu_1(i)={\mu(i)\over \mu(\maV_1)},\,i\in\maV_1\quad\mbox{ and }\quad\tilde\mu_2(j)={\mu(j)\over \mu(\maV_2)},\,j\in\maV_2,
\end{equation}
and by $\tilde\mu$, the probability measure of $\overline{\maM}(\maV_1\times\maV_2)$ having first (resp., second) marginal $\tilde\mu_1$ (resp., $\tilde\mu_2$). We have the following counterpart of Theorem \ref{thm:main} in the bipartite case, 
\begin{theorem}\label{thm:mainBip}
Let $G=(\maV,\maE)$ be a bipartite graph, and let $\mu\in\mathcal M(\maV)$. The following properties are equivalent.
\begin{enumerate}
\item[{\it (i)}] The measure $\mu$ is invariant for a reversible random walk on the edges of $G$.
\item[{\it (ii)}] The measure $\mu$ is a weighted measure, that is, there exist weights $(\alpha_{i,j})_{(i,j)\in\maE}$ such that for all $i\in\maV$,
$\mu(i) = \mu^\alpha(i)$.  
\item[{\it (iii)}] We have $\mu\left(\maU\right) < \mu\left(\maE\left(\maU\right)\right)$ for any non-empty subset $U$ of $\maV$ that is different from $\maV,\maV_1$ and $\maV_2$. 
\item[{\it (iv)}] The measure $\tilde\mu$ belongs to $\textsc{stab}(G,\textsc{ml})$, that is: 
if couple of items arrive according to the distribution $\tilde\mu$, then the extended bipartite matching models 
$(G,\textsc{ml},\tilde\mu)$ is stable. 
\item[{\it (v)}] The continuous-time bipartite matching queue of arrival intensities $\mu$ on servers and service intensities $\mu$ on customer nodes, 
of service discipline FCFS-ALIS, is stable. 
\end{enumerate}
\end{theorem}
Define, for any connected bipartite graph $G=(\maV_1\cup \maV_2,\maE)$, 
the set 
\begin{equation}
\label{eq:defNbG}
\mathcal{N}_{\textsc{b}}(G)\isdef
\left\{\mu \in \maM\left(\maV\right) : 
\left\{
    \begin{array}{ll}
        \forall \,\mathcal U\in\mathcal P(\maV)\setminus\{\emptyset,\maV,\maV_1,\maV_2\},\; 
        \mu\left(\maU\right) < \mu\left(\maE\left(\maU\right)\right) \\
        \mu(\maV_1) = \mu(\maV_2)
    \end{array}
\right.
\right\}.\end{equation}
Then, similarly to (\ref{eq:setequality}) the equivalence $(ii)\iff (iii)$ can be reformulated as 
\begin{equation}
\label{eq:setequalityBip}
\mathcal N_{\textsc{b}}(G) = \mathcal W(G),\mbox{ for any bipartite graph }G,
\end{equation}
a result that, again, can be easily exploited to the admission control of bipartite stochastic matching models, and matching queues: 
it is necessary and sufficient to set any family $\alpha$ of weights on the edges of $G$, 
$\mu^\alpha\in{\mathcal N}_{\textsc{b}}(G)$ in function of $\alpha$ according to (\ref{eq:deflambdaalpha}), and then 
to deduce $\tilde \mu^\alpha$ from $\mu^\alpha$ as in (\ref{eq:deftildemu}) to obtain an extended bipartite matching model that is stabilizable 
by the policy {\sc ml}, or a matching queue that is stabilizable by FCFS-ALIS. 
Moreover, as above an explicit representation of the family of weights 
$\alpha$ is provided 
as a simple function of the matching rates of the related extended bipartite model, 
see (\ref{eq:defThetaBip}), and Corollary \ref{coroll : NcondB(G) inclus dans l'ensemble des produits tensoriels de mesures a poids (cas biparti)}.
The proof of Theorem \ref{thm:mainBip} is provided in Section \ref{sec:proofmainBip}. 

\subsection{Asymmetric case}
In many applications, it is relevant to create an asymmetry between two sets of nodes. 
For a connected graph $G=(\maV,\maE)$, let $\maV_1\cup \maV_2$ be a (non-trivial) partition of $\maV$, 
and consider now the following set of measures,
$${\mathcal N}_{\maV_1}(G)=\left\{\mu\in\mathcal{M}(\mathcal{V}) : \forall \,\mathcal U_1\in \mathcal P(\maV_1)\setminus\{\emptyset\}, \; \mu(\maU_1) < \mu(\maE(\maU_1))\right\}.$$
The condition $\mu\in {\mathcal N}_{\maV_1}(G)$ is shown to be sufficient for the stability of a general stochastic matching model on a graph, ruled by a Max-Weight policy, in which nodes having classes in $\maV_2$ have a finite lifetime in the system, see Theorem 3.2 in \cite{JMRS20}. As will be demonstrated below, it also coincides with the stability region of a wide range of applications to skill-based queueing systems. The following result states that it can also be represented, and explicitly constructed, using weighted measures on $G$, 
\begin{theorem}\label{thm:mainqueue}
For any connected graph $G=(\maV,\maE)$ and any non-trivial partition $\maV_1\cup \maV_2$ of $\maV$, the set ${\mathcal N}_{\maV_1}(G)$ coincides with the set
$$\mathcal W_{\prec,\maV_2}(G) \isdef 
\left\{\mu\in\mathcal{M}(\mathcal{V}) : \exists \alpha\in{\mathcal{M}}(\maE) \; \text{s.t.} \;
\left\{
    \begin{array}{ll}
        \forall i\in\mathcal{V}_1, \;  \mu^\alpha(i) = \mu(i) \\
        \forall j\in\mathcal{V}_2,\; \mu^\alpha(j) < \mu(j)
    \end{array}
\right.
\right\}.$$
\end{theorem}
In Section~\ref{sec:proofqueueing}, two different proofs of Theorem~\ref{thm:mainqueue} are given. One using Theorem~\ref{thm:main} whereas the other is self-content and involves the so-called max-flow/min-cut theorem used in flow network theory. 
Whenever the restriction of the measure to the subset $\maV_2$ is fixed (say, equal to $\gamma$), the following result immediately follows, 
\begin{corollary}
\label{cor:mainqueue}
For any connected graph $G=(\maV,\maE)$ and any non-trivial partition $\maV_1\cup \maV_2$ of $\maV$, for any $\gamma\in \maM^+(\maV_2)$, 
the sets $${\mathcal N}_{\maV_1,\gamma}(G)\isdef\left\{\mu\in\mathcal{M}(\maV_1) : \forall \,\mathcal U_1\in \mathcal P(\maV_1)\setminus\{\emptyset\}, \; \mu(\maU_1) < \gamma(\maE(\maU_1))\right\}$$
and
$$\left\{\mu\in\mathcal{M}(\mathcal{V}_1) : \exists \alpha\in{\mathcal{M}}(\maE) \; \text{s.t.} \;
\left\{
    \begin{array}{ll}
        \forall i\in\mathcal{V}_1, \;  \mu^\alpha(i)=\mu(i) \\
        \forall j\in\mathcal{V}_2,\; \mu^\alpha(j) < \gamma(j)
    \end{array}
\right.
\right\}$$
coincide. 
\end{corollary}

We now focus again on the case where the graph $G=(\maV,\maE)$ is bipartite, of bipartition $\maV=\maV_1\cup \maV_2$. 
Consider a ``multi-class and multi-pool queue'' defined as follows: $\maV_1$ represents the set of classes of customers, and $\maV_2$ the set of types of parallel servers pools. For each servers pool $j\in\maV_2$, let us denote by $s_j\in\mathbb{N}_*$ the number of indistinguishable servers in $j$. 
To simplify the model under study, let us assume that each server can only serve one customer at a time. 
(Notice that this assumption is violated in various practical cases, see for instance \cite{GV18} and \cite{LZ13}.)  
To each class $i\in\maV_1$ of customers, such elements enter the system according to a Poisson point process of intensity $\mu(i)>0$. 
Any customer can potentially be served by more than one class of servers, and the bipartite graph $G=(\maV=\maV_1\cup\maV_2,\maE)$ represents the compatibilities between classes of servers and of customers. 
We suppose that the service times of customers of class $i\in\maV_1$, whenever served by server of class $j\in\maV_2$, is exponentially distributed 
of parameter $\gamma((i,j))>0$. We set $\gamma((i,j))=0$ if (and only if) $(i,j)\not\in\maE$. For details regarding this class of skill-based queueing systems, see e.g.~\cite{CDS20}. In that case, it is shown in \cite{TE93} that for any measure $\gamma\in\maM^+(\maE)$ and any $(s_j)_{j\in\maV_2}\in(\mathbb{N}_*)^{|\maV_2|}$, the set of measures 
\begin{multline*}
\widetilde{\mathcal N}_{\maV_1,\gamma,s}(G)\\
\isdef\left\{\mu\in\mathcal{M}(\mathcal{V}_1) : \exists \pi\in\maM^+(\maE)  \text{ s.t.} \;
\left\{
    \begin{array}{ll}
        \forall i\in\maV_1, \; \mu(i) = \sum\limits_{j\in\maV_2}\gamma((i,j))s_j\pi((i,j)) \\
        \forall j\in\maV_2, \; \sum\limits_{i\in\maV_1} \pi((i,j)) < 1
    \end{array}
\right.
\right\}\end{multline*}
is {\em maximal} for stability, in the sense that no system can reach ergodicity unless $\mu \in \widetilde{\mathcal N}_{\maV_1,\gamma,s}(G)$, within 
a set of admission controls (deciding the classes of servers the incoming customers are assigned to) and service policies (determining the classes of customers the available servers chose to serve) termed SBR policies, see ~\cite{CDS20}. Moreover, it is also shown in \cite{TE93} that the system is stable for any $\mu \in \widetilde{\mathcal N}_{\maV_1,\gamma,s}(G)$ whenever the so-called {\em Maximum Pressure Policy} $\textsc{mpp}$ is implemented - such a policy is then termed {\em throughput optimal}. Moreover, in that case, for any $(i,j)\in\maE$, $\pi((i,j))$ can be interpreted as the proportion of time in which servers of pool $j$ serve class $i$ customers, in steady state. 
First observe the following immediate characterization of the maximal stability region as a set of weighted measures on $\maV_1$,
\begin{proposition}\label{th:caracterisation de chapeau Ncond}
Let $G=(\maV_1\cup\maV_2,\mathcal{E})$ be a connected bipartite graph. Then, for any 
measure $\gamma\in\maM^+(\maE)$ and any $(s_j)_{j\in\maV_2}\in\left(\mathbb{N}_*\right)^{|\maV_2|}$, we have 
$$\widetilde{\mathcal N}_{\maV_1,\gamma,s}(G)=\left\{\mu\in\mathcal{M}(\maV_1) : \exists \alpha\in{\mathcal{M}}(\maE) \; \text{s.t.} \;
\left\{
    \begin{array}{ll}
        \forall i\in\maV_1, \;  \mu_\alpha(i) =  \mu(i) \\
        \forall j\in\maV_2,\; \sum\limits_{i\in\mathcal{E}(j)}\frac{\alpha_{i,j}}{\gamma((i,j))}< s_j
    \end{array}
\right.
\right\}.$$
\end{proposition}
\begin{proof}
For the left inclusion, it suffices to set $\alpha_{i,j}=\gamma((i,j))s_j\pi((i,j))$ for all $i\in\maV_1$ and $j\in\maV_2$. 
As for the right inclusion, just set $$\pi((i,j))=\frac{\alpha_{i,j}}{\gamma((i,j))s_j}\mathbf 1_{\{i-j\}},\quad i\in\maV_1,\,j\in\maV_2.$$
\end{proof}

In the particular case where service times depend only on the classes of servers and not on the classes of customers they serve (i.e., $\gamma((i,j))=\gamma(j)$ for all $(i,j)\in\maE$), by gathering Corollary \ref{cor:mainqueue} and Proposition \ref{th:caracterisation de chapeau Ncond}, we immediately obtain the following characterization of the maximal stability region 

\begin{proposition}
\label{prop:finalqueue}
Let $G=(\maV_1\cup\maV_2,\mathcal{E})$ be a connected bipartite graph, 
$\gamma\in\maM^+(\maV_2)$ and $(s_j)_{j\in\maV_2}\in\left(\mathbb{N}_*\right)^{|\maV_2|}$. Let $\hat\gamma\in\maM^+(\maV_2)$ and $\tilde \gamma \in \maM^+(\maE)$ be respectively defined by $\hat\gamma(j)\isdef s_j\gamma(j)$ for all $j\in\maV_2$, and by $\tilde \gamma((i,j))\isdef\gamma(j)$ for all $(i,j)\in\maE$. Then, the maximal stability region $\widetilde{\mathcal N}_{\maV_1,\tilde\gamma,s}(G)$,
the set ${\mathcal N}_{\maV_1,\hat\gamma}(G)$, and the set 
$$\left\{\mu\in\mathcal{M}(\mathcal{V}_1) : \exists \alpha\in{\mathcal{M}}(\maE) \; \text{s.t.} \;
\left\{
    \begin{array}{ll}
        \forall i\in\mathcal{V}_1, \;  \mu_\alpha(i)=\mu(i) \\
        \forall j\in\mathcal{V}_2,\; \mu_\alpha(j) < s_j\gamma(j)
    \end{array}
\right.
\right\}$$
coincide. 
\end{proposition}

\section{Proof of Theorem \ref{thm:main}}
\label{sec:proofmain}
We now turn to the proof of our main result, is the case where $G$ is not a bipartite graph.

\subsection{Proof of Theorem \ref{thm:main}}
Fix a connected multigraph $G$ that is not a bipartite graph. 
\subsection*{${(i) \iff (ii)}$} It is a folk result: to prove the implication $(ii) \implies (i)$, one can consider the weighted random walk associated to the weights $(\alpha_{i,j})_{(i,j)\in\maE}$, and conversely, if $\mu$ is a reversible invariant measure for the random walk of matrix $P$, then we can set $\alpha_{i,j}=\mu(i)P(i,j)=\mu(j)P(j,i)$ and check that $\mu$ is the weighted measure associated to the family of weights $(\alpha_{i,j})_{(i,j)\in\maE}$.

\subsection*{${(iii) \iff (iv)}$}  Implication $(iii) \implies (iv)$ is trivial. We now prove the converse implication, that generalizes Lemma 1 in~\cite{MaiMoy17} from graphs to multigraphs. 
Let $G=(\maV,\maE)$ be a connected multigraph and $\mu$ be a measure on $\maV$ satisfying 
$\mu\left(\mathcal I\right) < \mu\left(\maE\left(\mathcal I\right)\right)$ for any independent set $\mathcal I$ of $G$, 
and fix a subset $\maU\in\mathcal{P}(\maV)\backslash\{\emptyset,\maV\}$.
Denote in this proof $\maU_S \isdef\maU\cap \maV_\maS$ and $\maU_N \isdef\maU\cap \maV^c_\maS$, where $\maV_\maS$ is the set of nodes of $\maV$ having a self-loop. We get that 
\begin{equation*}
\mu\left(\maE(\maU)\right)-\mu(\maU) =\mu(\maE(\maU)\cap \maU_S)+\mu(\maE(\maU)\cap \maU_N)+\mu(\maE(\maU)\cap \maU^c)-\mu(\maU_S)-\mu(\maU_N).
\end{equation*}
But by the very definition of a self-loop, we get that $\maE(\maU)\cap \maU_S = \maU_S$, and thus the latter equality can be rewritten as
\begin{equation}\label{eq:truc1}
\mu\left(\maE(\maU)\right)-\mu(\maU) 
=\mu(\maE(\maU)\cap \maU^c)-\mu\left(\maU_N \cap \left(\maE(\maU)\right)^c\right).
\end{equation}
But any element in the neighborhood of $\maU_N \cap  
\left(\maE(\maU)\right)^c$ is in particular an element of 
$\maE(\maU)$, and also an element of $\maU^c$ as no element of $\maU$ can share an edge with an element of $\left(\maE(\maU)\right)^c$. Therefore, we get that 
\begin{equation}
\label{eq:truc2}
\maE\left(\maU_N \cap  
\left(\maE(\maU)\right)^c\right) \subset \maE(\mathcal U)\cap \maU^c.
\end{equation}
If $\maU_N \cap  
\left(\maE(\maU)\right)^c\neq\emptyset$, it is clearly an independent set, 
so we get that 
\begin{equation*}
\mu\left(\maU_N \cap  
\left(\maE(\maU)\right)^c\right) < \mu\left(\maE\left(\maU_N \cap  
\left(\maE(\maU)\right)^c\right)\right).
\end{equation*}
Gathering this with (\ref{eq:truc1}) and (\ref{eq:truc2}), we conclude that 
\[\mu\left(\maE(\maU)\right)-\mu(\maU)>0,\]
which completes the proof, since the case $\maU_N \cap  
\left(\maE(\maU)\right)^c=\emptyset$ is solved by~\eqref{eq:truc1}, the fact that $\maU\varsubsetneq\maV$ and the connectivity of $G$.

\subsection*{$(iv) \iff (v)$}
This equivalence is precisely Theorem 1 in \cite{BMMR21}. 

\subsection*{$(v) \iff (vi)$}
The equivalence between the stability of discrete-time and continuous-time general stochastic matching models was shown in Theorem 1 in 
\cite{MoyPer17}. 

\subsection*{$(ii) \implies (iv)$}
The implication $(ii) \implies (iv)$ was observed in Lemma 11 in \cite{Com21}, in the context of graphs. 
We can easily extend that proof to multigraphs, as follows. 
Let $G=(\maV,\maE)$ be a connected multigraph that is not a bipartite graph. 
For any subset $\mathcal{X}\subset\mathcal{V}$, denote by $\mathcal{A}_{\mathcal{X}}\subset \mathcal E$
the set of edges of the graph having at least one extremity in $\mathcal X$ (including self-loops $(i,i)$ for $i\in\mathcal X$). 
Let $\alpha\in\maM^+(\mathcal{E})$ be a family of weights, and fix an independent set $\mathcal{I}\in\mathbb{I}(G)$. Then, first observe that 
$\mathcal{A}_{\mathcal{I}}\varsubsetneq\mathcal{A}_{\mathcal{E}(\mathcal{I})}$. Indeed, for any 
$(a_1,a_2)\in\mathcal{A}_{\mathcal{I}}$, we have either $a_1 \in\maI$ and thus $a_2 \in \maE(\maI)$, 
or the other way around, and in both cases, $(a_1,a_2) \in \mathcal{A}_{\mathcal{E}(\mathcal{I})}.$ 
Therefore, we have that $\mathcal{A}_{\mathcal{I}}\subset\mathcal{A}_{\mathcal{E}(\mathcal{I})}.$ On another hand, as $G$ is connected and non-bipartite, there exists an edge $e$ connecting an element of $\mathcal{E}(\mathcal{I})$ to an element of $\mathcal I^c$, otherwise we would have $\maE(\maE(\maI))=\maI$ and, since $\maI\in\mathbb{I}(G)$, $(\maI,\maE(\maI))$ would form a bipartition of $G$.
      In particular, as $\maI$ is an independent set, we have that $e\in\mathcal{A}_{\mathcal{E}(\mathcal{I})}\cap \left(\mathcal{A}_{\mathcal{I}}\right)^c$, entailing that $\mathcal{A}_{\mathcal{I}}\varsubsetneq\mathcal{A}_{\mathcal{E}(\mathcal{I})}$. 
      We obtain that 
\begin{equation}
\label{eq:cagen}
\mu^\alpha(\mathcal{I}) 
=\sum\limits_{i\in\mathcal{I}}\sum\limits_{e\in\mathcal{A}_i}\alpha_e
=\sum\limits_{e\in\mathcal{A}_{\mathcal{I}}}\alpha_e
<\sum\limits_{e\in\mathcal{A}_{\mathcal{E}(\mathcal{I})}}\alpha_e 
\leq \sum\limits_{i\in\mathcal{E}(\mathcal{I})}\sum\limits_{e\in\mathcal{A}_i}\alpha_e 
=\mu^\alpha(\mathcal{E}(\mathcal{I})), 
\end{equation}
where the second equality holds due to the fact that the set $\maI$ is independent. 

\subsection*{$(iv) \implies (ii)$}
Let $G=(\maV,\maE)$ be a connected multigraph that is not a bipartite graph. We will show that the set of measures satisfying $(iv)$ is included in the set of weighted measures with positive or null weights. By taking the interior on each side of the inclusion (the set corresponding to $(iv)$ being already open), we will obtain the implication $(iv) \implies (ii)$.
 
Let $\mu\in\maM^+(\maV)$ be a measure. We will reason by contraposition and show that if $\mu$ is not a weighted measure with positive or null weights, then there exists an independent set $\mathcal I$ of $G$ such that $\mu\left(\mathcal I\right) \geq \mu\left(\maE\left(\mathcal I\right)\right)$.

Let $A$ be the matrix indexed by $\maV \times \maE$ such that for $(v,e)\in\maV \times \maE$, $A_{v,e}=1$ if the vertex $v$ is an extremity of the edge $e$, and $A_{v,e}=0$ otherwise. Introduce the vector $b=(\mu(v))_{v\in \maV}$, and for a given family of weights $\alpha\in{\mathcal{M}}^+(\maE)$ 
introduce the vector $x_{\alpha}=(\alpha_e)_{e\in \maE}$. With these notations, the measure $\mu$ is the weighted measure associated to the family of weights $\alpha$ if and only if $Ax_{\alpha}=b$. 

The rest of the proof will rely on Farkas' Lemma \cite{Far02}, that asserts that one and only one of the following linear systems has a solution:
\begin{enumerate}  
\item the system $Ax = b$ of unknown a column vector $x$ indexed by $\maV$ satisfying $x \geq 0$ ;
\item the system $^tA y \geq 0$ of unknown a column vector $y$ indexed by $\maE$ satisfying $^tb y < 0$.
\end{enumerate}

Let us consider the matrix $^tA$, indexed by $\maE \times \maV$. On the line indexed by the edge $e$,
\begin{itemize}
\item[-] if $e$ is not a self-loop, there are exactly two occurrences of the value $1$, at the positions corresponding to the two extremities of $e$,
\item[-] if $e$ is a self-loop, there is exactly one occurrence of the value $1$, at the position corresponding to the vertex having the self-loop $e$.
\end{itemize}
It follows that the equation $^tA y \geq 0$ is equivalent to: 
\begin{itemize}
\item[-] for any edge $e=(i,j)$ which is not a self-loop, $y(i)+y(j)\geq 0$, 
\item[-] for any self-loop $e=(i,i)$, $y(i)\geq 0$.
\end{itemize}

Let us assume that $\mu$ is not a weighted measure, that is, that the system $Ax = b, \; x\geq 0$ has no solution. Then, by Farkas' lemma, the system $^tA y \geq 0, \; ^tb y < 0$ has at least one solution, which means that there exists a vector $y$ satisfying $^tA y \geq 0$, and such that $\sum_{i\in \maV} \mu(i) y(i) <0$.

Let us choose a solution $y$ of the system such that $|y|$ takes as few different values as possible.

\begin{itemize}
\item If $|y|$ takes a single value $c$, then we must have $c>0$. Let $\maV_+\isdef\{i\in \maV : y(i)=c\}$ and $\maV_-\isdef\{i\in \maV : y(i)=-c\}$. In view of the constraints that $y$ must satisfy on the edges, one can check that $\maV_-$ is an independent set, and that $\E(\maV_-)\subset \maV_+$. Furthermore, the condition $\sum_{i\in \maV} \mu(i) y(i)<0$ implies that $\mu(\maV_+)<\mu(\maV_-)$, so that $\mu(\E(\maV_-))<\mu(\maV_-)$. Thus, the measure $\mu$ does not satisfy $(iv)$.
\item Otherwise, let $p$ be the largest value that $|y|$ takes, and let $q$ be the second largest value.
We introduce the sets $\maV^p_{+}\isdef\{i\in \maV : y(i)=p\}$ and $\maV^p_{-}\isdef\{i\in \maV : y(i)=-p\}$. 
Again, the constraints on the edges imply that $\maV^p_{-}$ is an independent set, and we have $\E(\maV^p_-)\subset \maV^p_+$, so that $\mu(\E(\maV^p_-))\leq \mu(\maV^p_+)$. 
\begin{itemize}
\item Let us first assume that $\mu(\maV^p_+)\leq\mu(\maV^p_-)$. Then, $\maV^p_-$ is an independent set such that $\mu(\E(\maV^p_-))\leq\mu(\maV^p_-)$, so that $\mu$ does not satisfy $(iv)$.
\item Let us now assume that $\mu(\maV^p_+)>\mu(\maV^p_-)$. Then, let us modify $y$ by giving the value $q$ to all the vertices of $\maV^p_+$ and the value $-q$ to all the vertices of $\maV^p_-$. One can check that the new vector $y'$ then defined still satisfy the constraints on the edges and that the new value of the sum $\sum_{i\in \maV} \mu(i) y(i)$ has decreased by $(p-q)(\mu(\maV^p_{+})-\mu(\maV^p_{-}))$, so that it is still negative. We have thus a contradiction with the fact that $|y|$ takes as few different values as possible.\end{itemize}
\end{itemize}
Gathering the above, we have that $(i) \iff (ii)$, $(iii) \iff (iv) \iff (v) \iff (vi)$, and $(ii) \iff (iv)$. Therefore, the proof of Theorem~\ref{thm:main} is complete.

\subsection{An alternative proof of $(v) \implies (ii)$ using matching rates}
\label{subsec:alternate}
We now provide an alternative proof of the implication $(v) \implies (ii)$, relying on the matching rates of a related stochastic matching model.  As will be shown in Section \ref{subsec:unique}, we can deduce from this characterization, interesting properties of the matching models at stake, among which, invariance properties of the matching rates with respect to the matching policy. 

\begin{definition}
For a general matching model associated to any fixed $(G,\Phi,\bar\mu)$, for any fixed time $n\geq 1$ and any fixed $i$ and $j$ in $\mathcal{V}$, we set 
$$M_n[i,j]\isdef\text{{\em \small{Number of matchings $(i,j)$ performed up to time $n$ included}}},$$
where the above is set as null a.s. if $(i,j)\not\in\mathcal{E}$. Then, the {\em matching rate} of $i$ with $j$ (or equivalently, of $j$ with $i$) up to $n$ is defined as the proportion $M_n[i,j] / n$. 
\end{definition}

\noindent In the sequel, for any fixed couple $(G,\Phi)$, any $i\in\mathcal{V}$ and any $j\in\mathcal{E}(i)$, we set 
\begin{equation}
\label{eq:defAij}
A_{i\rightarrow j}\isdef\{\text{\small{states of $\mathbb W$ s.t. an incoming $i$-item gets matched with a $j$-item}}\}.
\end{equation}

\begin{lemma}\label{prop : Mn[i,j]/n tend vers Theta[i,j]}
For any {stable} general matching model associated to $(G,\Phi,\bar\mu)$ and any $(i,j)\in \mathcal{E}$,  the asymptotic matching rates satisfy 
$$\frac{M_n[i,j]}{n}\xrightarrow[n\rightarrow +\infty]{\text{a.s.}} {\Theta}^{\Phi,\bar\mu}[i,j],$$
where
\begin{equation}\label{eq:Theta ij}
{\Theta}^{\Phi,\bar\mu}[i,j]\isdef \bar\mu(i)\Pi_W^{\Phi,\bar\mu}\left(A_{i\rightarrow j}\right) + \bar\mu(j)\Pi_W^{\Phi,\bar\mu}\left(A_{j\rightarrow i}\right)\mathbf 1_{\{i\neq j\}},
\end{equation}
and $\Pi_W^{\Phi,\bar\mu}$ represents the unique stationary distribution of the chain $\suite{W_n^{\Phi,\bar\mu}}$. 
\end{lemma}

\begin{proof}
Fix an edge $(i,j)\in\mathcal{E}$ with $i\neq j$. 
Then, we have a.s. 
\begin{align*}
\frac{M_n[i,j]}{n}
&{=} \frac{1}{n} \sum\limits_{k=1}^n \mathbf 1_{\{\text{a match $(i,j)$ is performed at time $k$}\}} \\
&{=} \frac{1}{n} \sum\limits_{k=1}^n \left[\mathbf 1_{\{V_k=i\}}\mathbf 1_{\left\{W_{k-1}^{\Phi,\mu}\in A_{i\rightarrow j}\right\}} + \mathbf 1_{\{V_k=j\}}\mathbf 1_{\left\{W_{k-1}^{\Phi,\mu}\in A_{j\rightarrow i}\right\}}\right] \\
&{= :} \frac{1}{n} \sum\limits_{k=0}^{n-1} f_{i,j}\left(W_k^{\Phi,\mu},V_{k+1}\right). 
\end{align*}
\noindent
Moreover, for all $k\in\mathbb{N}, \; W_k^{\Phi,\bar\mu} \; \text{and} \; V_{k+1}$ are independant, and 
$\suitek{W_k^{\Phi,\bar\mu}}$ is an ergodic Markov chain of stationary distribution $\Pi_W^{\Phi,\bar\mu}$. Thus, $\suitek{\left(W_k^{\Phi,\bar\mu},V_{k+1}\right)}$ is an ergodic Markov chain on $\mathbb{W}\times\mathcal{V}$, whose unique stationary distribution is given by $\Pi^{\Phi,\bar\mu}\isdef\Pi_W^{\Phi,\bar\mu}\otimes\bar\mu$. 
As $f_{i,j}$ is bounded by $1$, it is integrable with respect to $\Pi^{\Phi,\bar\mu}$ and the ergodic theorem for Markov chains gives that a.s.,
\begin{eqnarray*}
\underset{n\rightarrow +\infty}{\lim} \left(\frac{1}{n} \sum\limits_{k=0}^{n-1} f_{i,j}\left(W_k^{\Phi,\bar\mu},V_{k+1}\right)\right)&{=} &\sum\limits_{w\in\mathbb{W}}\sum\limits_{v\in\mathcal{V}} f_{i,j}(w,v)\Pi^{\Phi,\bar\mu}(w,v) \\
&{=} &\sum\limits_{w\in\mathbb{W}}\mathbf 1_{\{w\in A_{i\rightarrow j}\}}\Pi_W^{\Phi,\bar\mu}(w)\sum\limits_{v\in\mathcal{V}}\mathbf 1_{\{v=i\}}\bar\mu(v)\\
& & + \sum\limits_{w\in\mathbb{W}}\mathbf 1_{\{w\in A_{j\rightarrow i}\}}\Pi_W^{\Phi,\bar\mu}(w)\sum\limits_{v\in\mathcal{V}}\mathbf 1_{\{v=j\}}\bar\mu(v)\\
&{=}& \bar\mu(i)\Pi_W^{\Phi,\bar\mu}\left(A_{i\rightarrow j}\right) + \bar\mu(j)\Pi_W^{\Phi,\bar\mu}\left(A_{j\rightarrow i}\right),
\end{eqnarray*}
and the computation is the same without the second term, in the case $i=j$. 
\end{proof}

By combining the result above and the strong law of large numbers, we get the following,

\begin{lemma}\label{lemma : taux d'arrivee = taux de depart des i}
For any stable general matching model associated to $(G,\Phi,\mu)$, we have
\begin{equation}
\label{eq:defTheta}
\bar \mu(i)=\sum\limits_{j\in\mathcal{V}} \left(1+\mathbf 1_{\{j=i\}}\right){\Theta}^{\Phi,\bar\mu}[i,j],\quad \mbox{for all }i\in\maV. 
\end{equation}
\end{lemma}

\begin{proof}
Consider an item $i\in\mathcal{V}$ and denote, for all $n\in\N_*$, by $A_n(i)$ and $D_n(i)$ the number of arrivals and the number of departures up to time $n$ (included) of $i$-items, respectively. 
By the strong law of large numbers, we have 
\begin{equation*}
\frac{A_n(i)}{n} 
= \frac{1}{n}\sum\limits_{k=1}^n 1_{\{V_k=i\}} 
\xrightarrow[n\rightarrow +\infty]{\text{a.s.}} \mu(i).
\end{equation*}
Moreover, Lemma~\ref{prop : Mn[i,j]/n tend vers Theta[i,j]} entails that 
\begin{align*}
\frac{D_n(i)}{n} 
&= \frac{2M_n[i,i]}{n} + \sum\limits_{j\in\mathcal{V}\backslash\{i\}} \frac{M_n[i,j]}{n} \\
&\xrightarrow[n\rightarrow +\infty]{\text{a.s.}} 2{\Theta}^{\Phi,\bar\mu}[i,i] + \sum\limits_{j\in\mathcal{V}\backslash\{i\}} {\Theta}^{\Phi,\bar\mu}[i,j] = {\Theta}^{\Phi,\bar\mu}[i,i] + \sum\limits_{j\in\mathcal{V}} {\Theta}^{\Phi,\bar\mu}[i,j]. 
\end{align*}
Comparing to the two limits above, it remains to prove that 
$$\frac{A_n(i)-D_n(i)}{n}\xrightarrow[n\rightarrow +\infty]{\text{a.s.}} 0.$$ 
For this, observe that we have a.s., for all $n\ge 1$, 
\begin{equation}\label{eq : nombre de i dans Wn(omega) (cas general)}
\left|W_n^{\Phi,\bar\mu}\right|_i=A_n(i)-D_n(i)+\left|W_0^{\Phi,\bar\mu}\right|_i,
\end{equation}
where $|w|_i$ denotes the number of occurrences of $i$ in the word $w$.\\
Let us suppose that, 
on some event $\Omega'\subset\Omega$ with $\mathbb{P}(\Omega')>0$, we have 
\begin{equation*}
\underset{n\rightarrow +\infty}{\lim} \displaystyle\frac{A_n(i)-D_n(i)}{n}\neq 0.
\end{equation*}
Fix a realization $\omega\in\Omega'$. First, if $$\underset{n\rightarrow +\infty}{\lim} \displaystyle\frac{A_n(i)(\omega)-D_n(i)(\omega)}{n}> 0,$$ 
from \eqref{eq : nombre de i dans Wn(omega) (cas general)}, there would exist 
a rank $M(\omega)$ such that $\left|W_n^{\Phi,\mu}(\omega)\right|_i>0$ for all $n\ge M(\omega)$.  
Likewise, if we have that $$\underset{n\rightarrow +\infty}{\lim} \displaystyle\frac{A_n(i)(\omega)-D_n(i)(\omega)}{n}< 0,$$
then in view of \eqref{eq : nombre de i dans Wn(omega) (cas general)}, 
for some $M(\omega)$ we would have that $\left|W_n^{\Phi,\bar\mu}(\omega)\right|_i<0$ for all $n\ge M(\omega)$. All in all, we get that 
$\Omega'\subset \lim\inf_{n\to \infty} \left\{\omega\in\Omega : \left|W_n^{\Phi,\bar\mu}(\omega)\right|_i\ne0\right\}$, an absurdity in view of the recurrence of $\suite{W_n^{\Phi,\bar\mu}}$. This concludes the proof.
\end{proof}

We deduce the following result, from which the implication $(v) \implies (ii)$ of Theorem~\ref{thm:main} directly follows,

\begin{corollary}\label{cor:stabpoids}
For any connected multigraph $G=(\maV,\maE)$ and any admissible matching policy $\Phi$, we have
$$\textsc{stab}(G,\Phi)\subset\{\bar\mu^{\alpha}:\alpha\in{\mathcal{M}}(\maE)\}=\mathcal{W}^1(G).$$
Specifically, we have that 
$$\bar\mu\in\textsc{stab}(G,\Phi)\Longrightarrow \bar\mu=\bar\mu{\alpha^{\Phi,\mu}},$$
where
\begin{equation}
\label{eq:defalphamu}
\begin{array}{ccccc}
\alpha^{\Phi,\bar\mu} & : & \mathcal{E} & \longrightarrow & \mathbb{R}_* \\
 & & (i,j) & \longmapsto & \alpha^{\Phi,\bar\mu}((i,j))\isdef \left(1+\mathbf 1_{\{i=j\}}\right){\Theta}^{\Phi,\bar\mu}[i,j],
\end{array}
\end{equation}
for $\Theta^{\Phi,\bar\mu}$ defined by \eqref{eq:Theta ij}.
\end{corollary}

\begin{proof}
Let fix a connected multigraph $G=(\mathcal{V}, \mathcal{E})$ and a matching policy $\Phi$. 
Let $\bar\mu\in\textsc{stab}(G,\Phi)$ (if the set is empty, the result is trivial), in a way that the general stochastic matching model associated to $(G,\Phi,\mu)$ is stable. Then, for the family of weights $\alpha^{\Phi,\bar\mu}\in {\mathcal{M}}(\maE)$ defined by (\ref{eq:defalphamu}), recalling (\ref{eq:defmualpha}) and in view of Lemma \ref{lemma : taux d'arrivee = taux de depart des i}, we have that for all $i\in\maV$, 
\begin{equation*}
\mu{\alpha^{\Phi,\bar\mu}}(i)
= \frac{{\Theta}^{\Phi,\bar\mu}[i,i] + \sum\limits_{j\in\mathcal{V}}{\Theta}^{\Phi,\bar\mu}[i,j]}{\sum\limits_{l\in\mathcal{V}}\left({\Theta}^{\Phi,\bar\mu}[l,l] + \sum\limits_{j\in\mathcal{V}}{\Theta}^{\Phi,\bar\mu}[l,j]\right)} 
= \frac{\bar\mu(i)}{\sum\limits_{l\in\mathcal{V}}\bar\mu(l)} = \bar\mu(i), 
\end{equation*} 
which concludes the proof.
\end{proof}


\section{Proof of Theorem \ref{thm:mainBip}}
\label{sec:proofmainBip}

\subsection{Matching rates of extended bipartite matching models} 
Before we could turn to the proof of Theorem \ref{thm:mainBip}, it will be useful to 
adapt the arguments of Section \ref{subsec:alternate} to bipartite graphs. For this, we first define the matching rates in extended 
bipartite stochastic matching models, as follows, 
\begin{definition}
For any bipartite matching model associated to a triple $(G=(\mathcal{V}_1\cup\mathcal{V}_2,\mathcal{E}),\Phi,\tilde\mu)$, for any fixed time $n\geq 1$ and any fixed couple $(i,j)\in\mathcal{V}_1\times\mathcal{V}_2$, we set 
$$M_n^{\textsc{b}}[i,j]\isdef\text{{\em \small{Number of matchings $(i,j)$ performed up to time $n$ included}}},$$
where the above is set as null a.s. if $(i,j)\not\in\mathcal{E}$. 
Then, the {matching rate} of $i$ with $j$ up to $n$ is defined as the proportion $M_n^{\textsc{b}}[i,j] / n$.
\end{definition}

As above, for any fixed bipartite graph $G=(\maV_1\cup\maV_2,\maE)$ and any admissible $\Phi$, for any $i\in\maV_1$ and $j\in\maV_2$, we define the following sets:   
\begin{align*}
A^{\textsc{b}}_{i\rightarrow j} &\isdef\left\{\text{\small{buffers s.t. an incoming $i$-item gets matched with a stored $j$-item}}\right\},\\
A^{\textsc{b}}_{i\leftrightarrow j} &\isdef\{\text{\small{buffers s.t. an incoming 
$(i,j)$-couple gets matched together}}\}.
\end{align*}
(Observe that the latter sets are non-empty only if $(i,j)\in\maE$, 
and that the second one does not depend on $\Phi$.) 
The following result, whose proof is analog to that of Lemma \ref{prop : Mn[i,j]/n tend vers Theta[i,j]}, makes precise the asymptotic matching rates in a stable bipartite matching model, 

\begin{proposition}\label{prop : Mn[i,j]/n tend vers Theta[i,j] (cas biparti)}
Consider a bipartite matching model associated to the triple $(G=(\mathcal{V}_1\cup\mathcal{V}_2,\mathcal{E}),\Phi,\tilde\mu)$ such that 
$\tilde\mu\in \textsc{stab}_{\textsc{b}}(G,\Phi)$. Then, for any 
$(i,j)\in\maE$, we have that 
\begin{multline*}
\begin{aligned}
\frac{M_n^{\textsc{b}}[i,j]}{n} \xrightarrow[n\rightarrow +\infty]{\text{a.s.}}&\;\; {\Theta}_{\textsc{b}}^{\Phi,\tilde\mu}[i,j]\\
\isdef &\;\;\tilde\mu((i,\mathcal{V}_2\setminus\{j\}))\tilde{\Pi}^{\Phi,\tilde\mu}_W\left(A^{\textsc{b}}_{i\rightarrow j}\right) 
+ \tilde\mu((\mathcal{V}_1\setminus\{i\},j))\tilde{\Pi}^{\Phi,\tilde\mu}_W\left(A^{\textsc{b}}_{j\rightarrow i}\right)
\end{aligned}\\
+ \tilde\mu((i,j))\tilde{\Pi}^{\Phi,\mu}_W\left(A^{\textsc{b}}_{i\leftrightarrow j}\right),
\end{multline*}
where $\tilde{\Pi}^{\Phi,\tilde\mu}_W$ represents the unique stationary probability distribution of the positive recurrent Markov chain $\suite{Y_n^{\Phi,\tilde\mu}}$.
\end{proposition}

We can then adequate the arrival rates to the cumulative matching rates of the nodes. 
The following result can be proven similarly to 
Lemma \ref{lemma : taux d'arrivee = taux de depart des i}, 
\begin{lemma}\label{lemma : taux d'arrivee = taux de depart des i (cas biparti)}
For a bipartite matching model associated to any $(G=(\mathcal{V}_1\cup\mathcal{V}_2,\mathcal{E}),\Phi,\tilde \mu)$, such that 
$\tilde\mu=(\tilde\mu_1,\tilde\mu_2)\in \textsc{stab}_{\textsc{b}}(G,\Phi)$, we have that 
\begin{equation}
\label{eq:defThetaBip}\begin{cases}
\tilde\mu_1(i) &= \sum\limits_{j\in\mathcal{E}(i)} {\Theta}_{\textsc{b}}^{\Phi,\tilde\mu}[i,j],\mbox{ for all }i\in\maV_1;\\
\tilde\mu_2(i) &= \sum\limits_{j\in\mathcal{E}(i)} {\Theta}_{\textsc{b}}^{\Phi,\tilde\mu}[i,j],\mbox{ for all }i\in\maV_2.
\end{cases}\end{equation}
\end{lemma}
\noindent 
Let us introduce two following probability distributions, 
\begin{definition}
For any fixed connected bipartite graph $G=(\maV_1\cup\maV_2,\maE)$ and any fixed family of weights $\alpha\in\maM(\maE)$, we define the conditional probability measures 
$\tilde\mu^\alpha_1\in~\overline{\mathcal{M}}(\mathcal{V}_1)$ and $\tilde\mu^\alpha_2\in\overline{\mathcal{M}}(\mathcal{V}_2)$ as follows: for any $i\in \maV$, 
$$\tilde \mu^\alpha_1(i)\isdef\frac{\sum\limits_{j\in\mathcal{E}(i)}\alpha_{i,j}}{\sum\limits_{l\in\mathcal{V}_1}\sum\limits_{j\in\mathcal{E}(l)} \alpha_{l,j}}
,\,i\in\maV_1
\quad \text{ and } \quad
\tilde\mu^\alpha_2(i)\isdef\frac{\sum\limits_{j\in\mathcal{E}(i)}\alpha_{i,j}}{\sum\limits_{l\in\mathcal{V}_2}\sum\limits_{j\in\mathcal{E}(l)} \alpha_{l,j}}\mathbf ,\,i\in\maV_2.$$
\end{definition}
\noindent
As $(\mathcal{V}_1,\mathcal{V}_2)$ realize a bipartition of $\mathcal{V}$, we can deduce that 
\begin{equation}\label{eq:egalites mu alpha avec mu alpha 1 et mu alpha 2}
\forall i\in\mathcal{V}_1, \; \bar\mu^\alpha(i)=\frac{\tilde\mu^\alpha_1(i)}{2} \quad \text{and} \quad \forall j\in\mathcal{V}_2, \; \bar\mu^\alpha(j)=\frac{\tilde\mu^\alpha_2(j)}{2},
\end{equation}
and thus
\begin{equation}\label{eq:egalites mu alpha avec mu alpha 1 et mu alpha 2 (2)}
\bar\mu^\alpha(\maV_1)=\bar\mu^\alpha(\maV_2)=\frac{1}{2} \quad \text{and} \quad \mu^\alpha(\maV_1)=\mu^\alpha(\maV_2)=\frac{\mu^\alpha(\mathcal{V})}{2}.
\end{equation}

\begin{corollary}\label{coroll : NcondB(G) inclus dans l'ensemble des produits tensoriels de mesures a poids (cas biparti)}
Let $G=(\mathcal{V}_1\cup\mathcal{V}_2,\mathcal{E})$ be a connected bipartite graph. 
Then, for any admissible matching policy $\Phi$, we get that 
$$\textsc{stab}_{\textsc{b}}(G,\Phi)\subset\left\{\tilde\mu=(\tilde\mu^\alpha_1,\tilde\mu^\alpha_2)\in \overline{\maM}(\maV_1\times\maV_2): \alpha\in\maM^+(\maE)\right\}.$$
\end{corollary}

\begin{proof}
Fix a connected bipartite graph $G=(\mathcal{V}_1\cup\mathcal{V}_2, \mathcal{E})$ and an admissible policy $\Phi$, and let $\tilde\mu=(\tilde\mu_1,\tilde\mu_2) \in\textsc{stab}_{\textsc{b}}(G,\Phi)$, so that the bipartite matching model associated to $(G,\Phi,\tilde\mu)$ is stable. Let us also define the family of weights $\alpha^{\Phi,\tilde\mu}_{\textsc{b}}\in\maM^+(\maE)$, by
\begin{equation}
\label{eq:defalphamuBip}
\begin{array}{ccccc}
\alpha^{\Phi,\tilde\mu}_{\textsc{b}} & : & \mathcal{E} & \longrightarrow & \mathbb{R}_* \\
 & & (i,j) & \longmapsto & 
 {\Theta}_{\textsc{b}}^{\Phi,\tilde\mu}[i,j].
\end{array}\end{equation}
From Lemma \ref{lemma : taux d'arrivee = taux de depart des i (cas biparti)}, for all $i\in\maV_1$ we have that 
\begin{equation*}
\tilde \mu_1^{\alpha^{\Phi,\tilde\mu}_{\textsc{b}}}(i)
= \frac{\sum\limits_{j\in\mathcal{E}(i)}{\Theta}_{\textsc{b}}^{\Phi,\tilde\mu}[i,j]}{\sum\limits_{l\in\mathcal{V}_1}\sum\limits_{j\in\mathcal{E}(l)}{\Theta}_{\textsc{b}}^{\Phi,\tilde\mu}[l,j]}
= \frac{\tilde\mu_{1}(i)}{\sum\limits_{l\in\mathcal{V}_1}\tilde\mu_{1}(l)} 
= \tilde\mu_{1}(i).
\end{equation*} 
So we have $\tilde\mu_1=\tilde\mu_1^{\alpha^{\Phi,\tilde\mu}_{\textsc{b}}}$, and likewise $\tilde\mu_2=\mu_2^{\alpha^{\Phi,\tilde\mu}_{\textsc{b}}}$, which concludes the proof.
\end{proof}

\subsection{Proof of Theorem \ref{thm:mainBip}}
We can now turn to the proof of Theorem \ref{thm:mainBip}. 
Let $G=(\maV=\maV_1\cup\maV_2,\maE)$ be a connected bipartite graph. 
First observe that the equivalence ${(i) \iff (ii)}$ proven in Theorem \ref{thm:main} is not peculiar to the non-bipartite case, and remains valid 
in the present context. To complete the proof, we then show that ${(iii) \iff (iv)}$, ${(iv) \iff (v)}$, and then ${(ii) \implies (iii)}$ and ${(iv) \implies (ii)}$. 

\subsection*{${(iii) \iff (iv)}$}
We first show that 
\begin{equation}
\label{eq:contact}
\mathcal{N}_{\textsc{b}}(G) = 
\left\{\mu \in \maM\left(\maV\right) : 
\left\{
    \begin{array}{ll}
        \forall \,\maU_1\in\mathcal P(\maV_1)\setminus\{\emptyset,\maV_1\}, \; 
        \mu\left(\maU_1\right) < \mu\left(\maE\left(\maU_1\right)\right) \\
        \forall \,\maU_2\in\mathcal P(\maV_2)\setminus\{\emptyset,\maV_2\}, \; 
        \mu\left(\maU_2\right) < \mu\left(\maE\left(\maU_2\right)\right) \\
        \mu(\maV_1) = \mu(\maV_2)
            \end{array}
\right.
\right\}.
\end{equation}
The left inclusion being trivial, let us focus on the converse. 
Let $\mu$ be an element of the right-hand set, and let $\maU$ be a non-empty set that is strictly included in $\maV$ and different from $\maV_1$ and $\maV_2$. 
Let us first observe that it cannot be the case that $\maU\cap\maV_1=\maV_1$ and $\maU\cap\maV_2=\maV_2$, otherwise 
$\maU$ would coincide with $\maV$. Thus, assuming for instance that $\maU\cap\maV_2$ is strictly included in $\maV_2$ (the other case is symmetric), 
we have that $\mu\left(\maU\cap\maV_2\right) < \mu\left(\maE\left(\maU\cap\maV_2\right)\right)$. On the other hand, we also have that 
$\mu\left(\maU\cap\maV_1\right) < \mu\left(\maE\left(\maU\cap\maV_1\right)\right)$ whenever  
$\maU\cap\maV_1$ is strictly included in $\maV_1$, 
while, in the case where $\maU\cap\maV_1=\maV_1$, we get that 
$$\mu\left(\maU\cap\maV_1\right) =\mu(\maV_1)=\mu(\maV_2)= \mu(\maE(\maV_1))=\mu\left(\maE\left(\maU\cap\maV_1\right)\right).$$ 
We obtain in all cases that 
\begin{align*}
\mu(\maU) = \mu\left(\maU\cap\maV_1\right) +  \mu\left(\maU\cap\maV_2\right)
	   	        &< \mu\left(\maE\left(\maU\cap\maV_1\right)\right) + \mu\left(\maE\left(\maU\cap\maV_2\right)\right)\\
		        &= \mu\left(\maE\left(\maU\right)\cap\maV_2\right) + \mu\left(\maE\left(\maU\right)\cap\maV_1\right)=\mu(\maE\left(\maU\right)),
\end{align*}
hence $\mu$ is an element of $\mathcal{N}_{\textsc{B}}(G)$, which completes the proof of (\ref{eq:contact}).
In particular, Assertion $(iii)$ is equivalent to saying that $\mu$ belongs to the right-hand set of (\ref{eq:contact}). 

Now, from Theorem 7.1 of \cite{BGM13} we have that 
\begin{equation}
\label{eq:maxmlB}
\textsc{stab}_{\textsc{b}}(G,\textsc{ml})=\tilde{\mathcal{N}}_{\textsc{b}}(G),
\end{equation}
where 
\begin{multline*}
\tilde{\mathcal{N}}_{\textsc{b}}(G) \isdef 
\Biggl\{\tilde\mu=(\tilde\mu_1,\tilde\mu_2)\in \overline{\maM}\left(\maV_1\times\maV_2\right) :\\
   \left\{ \begin{array}{ll}
        \forall \maU_1\in\mathcal{P}(\maV_1)\backslash\{\emptyset,\maV_1\}, \; 
        \tilde\mu_1\left(\maU_1\right) < \tilde\mu_2\left(\maE\left(\maU_1\right)\right) \\
        \forall \maU_2\in\mathcal{P}(\maV_2)\backslash\{\emptyset,\maV_2\}, \; 
        \tilde\mu_2\left(\maU_2\right) < \tilde\mu_1\left(\maE\left(\maU_2\right)\right)
            \end{array}
\right.
\Biggl\}.
\end{multline*}
and we conclude by observing that $\mu\in\mathcal{N}_{\textsc{b}}(G)$ is equivalent to $\tilde \mu \in \tilde{\mathcal{N}}_{\textsc{b}}(G)$.

\subsection*{${(iv) \iff (v)}$} In view of the characterization \eqref{eq:contact}, the equivalence ${(iv) \iff (v)}$ is shown in Theorem 1 in \cite{AW14}. 

\subsection*{${(ii) \implies (iii)}$} 
Fix $\alpha\in\maM(\maE)$, and recall that for any subset $\mathcal{X}\subset\mathcal{V}$, we denote by $\mathcal{A}_{\mathcal{X}}\subset \mathcal E$ the set of edges of the graph having at least one extremity in $\mathcal X$. It suffices to show that $\mu^\alpha$ belongs to the right-hand side of (\ref{eq:contact}). For this, fix a set $\mathcal U_1\in\mathcal{P}(\mathcal V_1)\backslash\{\emptyset,\maV_1\}$. 
As in the proof of $(ii)\implies(iv)$ in Theorem \ref{thm:main}, it is immediate that 
$\mathcal{A}_{\mathcal{U}_1}\subset\mathcal{A}_{\mathcal{E}(\mathcal{U}_1)}$. 
On another hand, by the connectivity of $G$, there exists an edge $e$ 
connecting an element of $\mathcal{E}(\mathcal{U}_1)$ to an element of 
$\mathcal U_1^c\cap\maV_1$, otherwise we would have $\maE(\maE(\maU_1))=\maU_1$ and, since $\maU_1\in\mathbb{I}(G)$, $(\mathcal{U}_1,\mathcal{E}(\mathcal{U}_1))$ would form a bipartition of $G$, which is absurd since $\mathcal{U}_1\neq\mathcal{V}_1$. So we have that $\mathcal{A}_{\mathcal{U}_1}\varsubsetneq\mathcal{A}_{\mathcal{E}(\mathcal{U}_1)}$. Since $\maU_1$ is an independent set, the same inequalities as in (\ref{eq:cagen}) hold and we have that $\mu^\alpha(\mathcal{U}_1)<\mu^\alpha(\mathcal{E}(\mathcal{U}_1)).$

By a symmetric argument, we also have that $
\mu^\alpha(\mathcal{U}_2) 
<\mu^\alpha(\mathcal{E}(\mathcal{U}_2))$, for any set $\mathcal U_2\in\mathcal{P}(\mathcal V_2)\backslash\{\emptyset,\maV_2\}$. Using~\eqref{eq:egalites mu alpha avec mu alpha 1 et mu alpha 2 (2)}, this shows that $\mu^\alpha \in\mathcal N_{\textsc{b}}(G)$.

\subsection*{${(iv) \implies (ii)}$} 
Fix a measure $\tilde\mu\in \textsc{stab}_{\textsc{b}}(G,\textsc{ml})$. 
Gathering Corollary \ref{coroll : NcondB(G) inclus dans l'ensemble des produits tensoriels de mesures a poids (cas biparti)} for $\Phi=\textsc{ml}$ with (\ref{eq:maxmlB}), 
we obtain that 
\begin{equation}\label{eq:inclusion tilde N_B^1(G)}
\tilde{\mathcal N}_{\textsc{b}}(G)\subset\left\{\tilde\mu=(\tilde\mu^\alpha_1,\tilde\mu^\alpha_2)\in \overline{\maM}(\maV_1\times\maV_2): \alpha\in\maM^+(\maE)\right\}.
\end{equation}
For any $\mu\in\mathcal N_{\textsc{b}}(G)$, 
we get that $\tilde\mu\in\tilde{\mathcal N}_{\textsc{b}}(G)$, so in view of \eqref{eq:inclusion tilde N_B^1(G)}, we obtain that $\tilde\mu_1=\tilde\mu^\alpha_1$ and $\tilde\mu_2=\tilde\mu^\alpha_2$, for some $\alpha \in \maM^+(\maE)$.
In view of \eqref{eq:egalites mu alpha avec mu alpha 1 et mu alpha 2}, and the fact that $\mu\in\mathcal N_{\textsc{b}}(G)$, we get that $\mu=\mu(\mathcal{V})\mu^{\alpha}$, which amounts to writing that $\mu=\mu^{\check{\alpha}}$ where $\check{\alpha}=\frac{\mu(\mathcal{V})}{\mu{\alpha}(\mathcal{V})}\alpha$, and so the proof is concluded.

\section{Proof of Theorem \ref{thm:mainqueue}}
\label{sec:proofqueueing}
%
\subsection{Proof using Theorem~\ref{thm:main}}
We start with the following result,
\begin{proposition}\label{prop:mainqueue}
For any connected graph $G=(\maV,\maE)$ and any non-trivial partition $\maV_1\cup \maV_2$ of $\maV$, we have
$${\mathcal N}_{\maV_1}(G) = \mathcal W(\hat G),$$
where $\hat G=(\maV,\hat{\maE})$ is the multigraph obtained from $G$ by adding a self-loop at any element of $\maV_2$. 
\end{proposition} 
\begin{proof}[Proof of Proposition \ref{prop:mainqueue}]
In view of Theorem \ref{thm:main}, it suffices to show that 
\begin{equation}
\label{eq:truc0}
{\mathcal N}_{\maV_1}(G)={\mathcal N}(\hat G).
\end{equation}
The right inclusion in (\ref{eq:truc0}) is immediate: if we let $\mu\in{\mathcal N}(\hat G)$, and fix a non-empty subset $\maU_1 \subset \maV_1$, then 
as $\maV_2\ne \emptyset$, we get that $\maU_1 \subsetneq \maV$ and thus we readily obtain that 
$$\mu(\maU_1) <\mu (\hat\maE(\maU_1))=\mu(\maE(\maU_1)),$$
where the second equality follows from the definition of $\hat\maE$. Hence, $\mu \in {\mathcal N}_{\maV_1}(G)$. 
We now turn to the left inclusion. Let $\mu\in {\mathcal N}_{\maV_1}(G)$ and 
fix a set $\mathcal U \in\mathcal{P}(\maV)\backslash\{\emptyset,\maV\}$.
Denote in this proof, $\maU_1\isdef\maU\cap \maV_1$ and $\maU_2\isdef\maU\cap \maV_2$. 
Reasoning exactly as in the argument leading to (\ref{eq:truc1}) and (\ref{eq:truc2}), we obtain that 
\begin{align}\label{eq:truc3}
\mu(\hat\maE(\maU))-\mu(\maU) 
&=\mu(\hat\maE(\maU)\cap \maU^c)-\mu(\maU_1 \cap (\hat\maE(\maU))^c);\\
\label{eq:truc4}
\hat\maE(\maU_1 \cap  
(\hat\maE(\maU))^c) &\subset \hat\maE(\mathcal U)\cap \maU^c.
\end{align}
Now, by the very definition of $\hat G$ we clearly have that 
\[\maU_1 \cap  (\hat\maE(\maU))^c= \maU_1 \cap  (\maE(\maU))^c\quad\mbox{ and }\quad\hat\maE(\maU_1 \cap  
(\hat\maE(\maU))^c)=\maE(\maU_1 \cap  (\maE(\maU))^c).\] 
Consequently, if $\maU_1 \cap  
\left(\maE(\maU)\right)^c\neq\emptyset$, as $\mu$ is an element of 
${\mathcal N}_{\maV_1}(G)$ we get that  
\begin{equation*}
\mu(\maU_1 \cap  (\hat\maE(\maU))^c)= \mu\left(\maU_1 \cap  
\left(\maE(\maU)\right)^c\right) < \mu\left(\maE\left(\maU_1 \cap  \left(\maE(\maU)\right)^c\right)\right)=\mu(\hat\maE(\maU_1 \cap  
(\hat\maE(\maU))^c)).
\end{equation*}
Gathering this with (\ref{eq:truc3}) and (\ref{eq:truc4}), we conclude that 
\[\mu(\hat\maE(\maU))-\mu(\maU)>0,\]
which completes the proof of the left inclusion in (\ref{eq:truc0}), since the case $\maU_1 \cap \left(\maE(\maU)\right)^c=\emptyset$ is solved by~\eqref{eq:truc3}, the fact that $\maU \subsetneq \maV$ and the connectivity of $\hat G$.
\end{proof}

For any finite set $A$, for two measures $\mu,\mu'$ of $\mathcal{M}(A)$, and for a subset $B \subset A$, we denote $\mu \prec_B \mu'$ whenever $\mu(i)<\mu'(i)$ for any $i\in B$, and 
$\mu(i)=\mu'(i)$ for any $i\in A\setminus B$. 
For any $\maV'\subset \maV$, we define the set 
\begin{align*}
\mathcal W_{\prec,\maV'}(G) &\isdef\left\{ \mu\in \maM(\maV)\,:\, \mu^\alpha\prec_{\maV'} \mu\mbox{ for some }\alpha\in\mathcal{M}(\maE)\right\}.
\end{align*}

\noindent We can now prove Theorem \ref{thm:mainqueue},
\begin{proof}[Proof of Theorem \ref{thm:mainqueue}]
From Proposition \ref{prop:mainqueue}, it is enough to show that 
$$\mathcal W(\hat G)=\mathcal W_{\prec,\maV_2}(G).$$
Regarding the left inclusion, for any $\mu\in \mathcal W(\hat G)$, letting $\hat \alpha \in \mathcal M^+(\hat {\mathcal E})$ be such that 
$\mu=\mu^{\hat \alpha}$, and $\alpha\in\mathcal M^+(\mathcal E)$ be the restriction of $\hat\alpha$ to $\mathcal E$ we get that
$\mu(i)=\mu^{\hat\alpha}(i)=\mu^\alpha(i)$ for all $i\in\maV_1$, while for any $i\in\maV_2$, 
we obtain $$\mu(i)=\mu^{\hat{\alpha}}(i) = \hat\alpha_{i,i}+ \sum_{j\ne i:\,j\in\maE(i)}\hat\alpha_{i,j}>\sum_{j\ne i:\,j\in\maE(i)}\alpha_{i,j}=\mu^\alpha(i),$$
hence $\mu \in  \mathcal W_{\prec,\maV_2}(G).$ As for the converse, let 
$\mu \in  \mathcal W_{\prec,\maV_2}(G)$, and $\alpha\in\mathcal M^+(\maE)$ be such that 
$\mu^\alpha \prec_{\maV_2} \mu$. Then, setting $\hat\alpha_{i,j}=\alpha_{i,j}$ for all $i\ne j$ and $\hat\alpha_{i,i} = \mu(i)-\mu^\alpha(i)$ for all 
$i\in\maV_2$, we easily retrieve that $\mu=\mu^{\hat\alpha}$. Hence $\mu\in \mathcal W(\hat G)$, which completes the proof. 
\end{proof}

\subsection{Bipartite case: proof using flow theory}
Interestingly enough, in the particular case where $G$ is bipartite of bipartition $\maV_1 \cup \maV_2$, 
Theorem \ref{thm:mainqueue} is in fact reminiscent of a simple flow argument. Hereafter, we give an alternative and independent proof 
of this result in the bipartite case, using flow network theory, see \cite{FF15}. 
We start with the following result,
\begin{proposition}\label{prop:flow graph}
For any connected bipartite graph $G=(\maV=\maV_1\cup \maV_2,\maE)$ and any measure $\mu\in\mathcal{N}_{\mathcal{V}_1}^+(G)$, there exists a $s-t$ flow graph associated to $G$, denoted by $\stackrel{\rightarrow}{G}\;=\left(\stackrel{\rightarrow}{\mathcal{V}},\stackrel{\rightarrow}{\mathcal{E}}\right)$, and a $s-t$ flow of $\stackrel{\rightarrow}{G}$, denoted by $\overline{f}$, such that:
\begin{equation}\label{eq:egalite f barre}
\sum\limits_{w\in\stackrel{\rightarrow}{\mathcal{V}}} \overline{f}(s,w)
=\mu(\mathcal{V}_1).
\end{equation}
\end{proposition}

\begin{proof}[Proof of Proposition~\ref{prop:flow graph}]
Let consider a connected bipartite graph $G=(\maV=\maV_1\cup \maV_2,\maE)$ and a measure $\mu\in\mathcal{N}_{\mathcal{V}_1}^+(G)$.\\ It then allows to consider a measure $\delta\in\mathcal{M}(\mathcal{V}_2)$ such that:
\begin{equation}\label{eq:construction du vecteur Delta}
\left\{
    \begin{array}{ll}
        \forall \,\mathcal U_1\in \mathcal P(\maV_1)\setminus\{\emptyset\}, \; \mu(\maU_1) < \delta(\maE(\maU_1)) ; \\
        \forall j\in\mathcal{V}_2,\; \delta(j)<\mu(j).
    \end{array}
\right.
\end{equation}
Now, let construct a $s-t$ flow graph $\stackrel{\rightarrow}{G}\;=\left(\stackrel{\rightarrow}{\mathcal{V}},\stackrel{\rightarrow}{\mathcal{E}}\right)$ associated to $G$ as follows:\\
- set a "source", denoted by $s$, and add the edges between $s$ and each vertex of $\maV_1$;\\
- set a "well", denoted by $t$, and add the edges between each vertex of $\maV_2$ and $t$;\\
- for all $i\in\maV_1$ and $j\in\maV_2$, let orient the edges from $s$ to $i$, from $j$ to $t$ and from $i$ to $j$, for all edge $(i,j)\in\mathcal{E}$;\\
- let associate a capacity $c(.,.)$ to each edge as follows: for all $i\in\maV_1$ and $j\in\maV_2$, we set $c(s,i)\isdef\mu(i)$, $c(j,t)\isdef\delta(j)$ and $c(i,j)\isdef +\infty$, for all edge $(i,j)\in\mathcal{E}$.\\
\\
In order to use the max-flow/min-cut theorem, we need the following result,

\begin{lemma}\label{lemma:minimal s-t cut of G}
By defining $\widetilde{\mathscr{S}}\isdef\{s\}$ and $\widetilde{\mathscr{T}}\isdef \;\stackrel{\rightarrow}{\mathcal{V}}\backslash\{s\}$, we have:
$$\underset{(\mathscr{S},\mathscr{T})\in\mathfrak{C}_{s,t}\left(\stackrel{\rightarrow}{G}\right)}{\min} c(\mathscr{S},\mathscr{T})
=c\left(\widetilde{\mathscr{S}},\widetilde{\mathscr{T}}\right)
=\mu(\mathcal{V}_1),$$
where $\mathfrak{C}_{s,t}\left(\stackrel{\rightarrow}{G}\right)$ is the set of all the $s-t$ cuts of $\stackrel{\rightarrow}{G}$.
\end{lemma}

\begin{proof}[Proof of Lemma~\ref{lemma:minimal s-t cut of G}]
By considering a $s-t$ cut $\left(\mathscr{S},\mathscr{T}\right)$ of $\stackrel{\rightarrow}{G}$ and with the help of~\eqref{eq:construction du vecteur Delta}, the proof consists in verifying that $c\left(\widetilde{\mathscr{S}},\widetilde{\mathscr{T}}\right)\leq c(\mathscr{S},\mathscr{T})$ in each following case:\\
- $\mathcal{V}_1\cap\mathscr{S}=\emptyset$ and $\mathcal{V}_2\cap\mathscr{S}=\emptyset$;\\
- $\mathcal{V}_1\cap\mathscr{S}\neq\emptyset$ and [$(\mathcal{V}_2\cap\mathscr{T})\cap\mathcal{E}(\mathcal{V}_1\cap\mathscr{S})\neq\emptyset$ or $(\mathcal{V}_2\cap\mathscr{S})\cap\mathcal{E}(\mathcal{V}_1\cap\mathscr{T})\neq\emptyset$];\\
- $\mathcal{V}_1\cap\mathscr{S}\neq\emptyset$ and [$(\mathcal{V}_2\cap\mathscr{T})\cap\mathcal{E}(\mathcal{V}_1\cap\mathscr{S})=\emptyset$ and $(\mathcal{V}_2\cap\mathscr{S})\cap\mathcal{E}(\mathcal{V}_1\cap\mathscr{T})=\emptyset$];\\
- $\mathcal{V}_2\cap\mathscr{S}\neq\emptyset$ and [$(\mathcal{V}_1\cap\mathscr{T})\cap\mathcal{E}(\mathcal{V}_2\cap\mathscr{S})\neq\emptyset$ or $(\mathcal{V}_1\cap\mathscr{S})\cap\mathcal{E}(\mathcal{V}_2\cap\mathscr{T})\neq\emptyset$];\\
- $\mathcal{V}_2\cap\mathscr{S}\neq\emptyset$ and [$(\mathcal{V}_1\cap\mathscr{T})\cap\mathcal{E}(\mathcal{V}_2\cap\mathscr{S})=\emptyset$ and $(\mathcal{V}_1\cap\mathscr{S})\cap\mathcal{E}(\mathcal{V}_2\cap\mathscr{T})=\emptyset$].
\end{proof}

The proof of Proposition~\ref{prop:flow graph} is then concluded by combining Lemma~\ref{lemma:minimal s-t cut of G} with the well-known max-flow/min-cut theorem.
\end{proof}

\noindent We are then able to retrieve the proof of Theorem~\ref{thm:mainqueue} in the bipartite case, 

\begin{proof}[Proof of Theorem~\ref{thm:mainqueue} for $G$ bipartite]
Let us consider a connected bipartite graph $G=(\maV=\maV_1\cup \maV_2,\maE)$. 
In order to prove the first inclusion, we fix a measure $\mu\in\mathcal{N}_{\mathcal{V}_1}^+(G)$, and define the non-negative measure $\alpha^{\overline{f}}\in\overline{\mathcal{M}}(\mathcal{E})$, by
$$\forall (i,j)\in\mathcal{E},\; \alpha^{\overline{f}}_{i,j}\isdef\overline{f}(i,j),$$
where $\overline{f}$ is the flow of $\stackrel{\rightarrow}{G}$, given by Proposition~\ref{prop:flow graph}. 
On the one hand, by the very structure of $\stackrel{\rightarrow}{G}$ and from~\eqref{eq:construction du vecteur Delta}, we have that 
$$\forall j\in\mathcal{V}_2,\; \mu^{\alpha^{\overline{f}}}(j)=\displaystyle\sum\limits_{i\in\mathcal{V}_1} \alpha^{\overline{f}}_{i,j} < \mu(j).$$
On the other hand, it can be proven that 
\begin{equation}\label{eq : f barre(s,u)=lambda_u}
\forall i\in\mathcal{V}_1,\; \overline{f}(s,i)=\mu(i).
\end{equation}
Indeed, the inequality ``$\leq$'' follows from the capacity constraint, whereas the converse inequality follows by contradiction, 
using ~\eqref{eq:egalite f barre}. 
It then follows from~\eqref{eq : f barre(s,u)=lambda_u} and by the very definition of $\alpha^{\overline{f}}$, that 
$$\forall i\in\mathcal{V}_1, \mu^{\alpha^{\overline{f}}}(i)=\sum\limits_{j\in\mathcal{V}_2} \alpha^{\overline{f}}_{i,j} = \mu(i).$$
In order to prove the converse inclusion, let us consider a measure $\mu\in\mathcal{W}^+_{\prec,\maV_2}(G)$. 
Then there exists a non-negative measure $\alpha\in{\mathcal{M}}^{\ge 0}(\mathcal{E})$ such that
\begin{equation}\label{eq : conditions sur le poids alpha}
\left\{
    \begin{array}{ll}
    \forall i\in\mathcal{V}_1,\; \mu^\alpha(i)=\mu(i) \\
    \forall j\in\mathcal{V}_2,\; \mu^\alpha(j) < \mu(j)
    \end{array}.
\right.
\end{equation}
Let us consider a non-empty subset $\mathcal{U}_1\in\mathcal{P}(\mathcal{V}_1)\backslash\{\emptyset\}$. 
Then, by~\eqref{eq : conditions sur le poids alpha} and since $\alpha\in{\mathcal{M}}^{\ge 0}(\mathcal{E})$, we have that 
$$\mu(\mathcal{U}_1) 
=\mu^\alpha(\mathcal{U}_1)
=\sum\limits_{j\in\mathcal{E}(\mathcal{U}_1)}\sum\limits_{i\in\mathcal{U}_1}\alpha_{i,j}
\leq \sum\limits_{j\in\mathcal{E}(\mathcal{U}_1)}\sum\limits_{i\in\mathcal{V}_1}\alpha_{i,j} 
=\mu{\alpha}(\mathcal{E}(\mathcal{U}_1)) \\
<\mu(\mathcal{E}(\mathcal{U}_1)),$$
which concludes the proof.
\end{proof}

\section{About uniqueness of the matching rates}
\label{sec:appli}

In Lemma \ref{lemma : taux d'arrivee = taux de depart des i} (resp., Lemma \ref{lemma : taux d'arrivee = taux de depart des i (cas biparti)}), 
we have proven that the 
arrival rates of stable general (respectively, bipartite) stochastic matching models are given as explicit functions of the matching rates of the various edges - see \eqref{eq:defTheta} (resp., \eqref{eq:defThetaBip}). 
Specifically, recalling \eqref{eq:deftildemu}, 
both systems of equations \eqref{eq:defTheta} and \eqref{eq:defThetaBip} can be rewritten 
under the generic form 
\begin{equation}
\label{eq:S}
\tag{S}
\bar\mu(i)=\sum\limits_{j\in\mathcal{E}(i)} \alpha_{i,j},\quad i\in \maV,
\end{equation}
where the weights $\alpha_{i,j}$, $i,j\in \maE$, represent the matching rates in the various cases. 
These weights depend on the matching policy, and are not unique in general. However, for various (multi-) graph geometries, they are uniquely defined (and thereby, {\em independent of of the matching policy}), as we will show hereafter. 

In what follows, for a rooted tree $G=(\maV,\maE,r)$ of root $r$ having at least two nodes, we denote by $d:=d(G)$, the depth of $G$. 
For any $\ell \in \llbracket 1,d \rrbracket$, we let $n_\ell$ be the number of elements of generation $\ell$, and label by 
$\ell_1,...,\ell_{n_\ell}$, the nodes of generation $\ell$ in an arbitrary manner. In that way, nodes $d_{1},...,d_{n_d}$ are the leaves of $G$. 
 For any $\ell\in\llbracket 1,d \rrbracket$ and $i\in\llbracket 1,n_\ell\rrbracket$, we denote by $f(\ell_i)$, the father of node $\ell_i$. Notice that by construction, we have $f(1_i)=r$ for all $i\in\llbracket 1,n_1 \rrbracket$. Last, for any $\ell\in\llbracket 1,d \rrbracket$ 
 and $i\in\llbracket 1,n_\ell\rrbracket$, we denote by $\mathscr S(\ell_i)$ and $\mathscr D(\ell_i)$, the sets of sons and of descendants of $\ell_i$, respectively. 
 Notice that the above sets are empty if, and only if $\ell=d$, i.e. $\ell_i$ is a leaf.

\begin{lemma}
\label{lemma:solrootedtree}
Let $G=(\maV,\maE,r)$ be a rooted tree having at least two nodes, and $\bar\mu\in\overline{\maM}(G)$.  
Then, the system of equations 
 \begin{equation}
\label{eq:Sstar}
\tag{S$^*$}
\bar\mu(i)=\sum\limits_{j\in\mathcal{E}(i)} \alpha_{i,j},\quad i\in \maV\setminus\{r\}
\end{equation}
admits the following unique solution on $\overline{\mathcal W}(G)$:
\begin{equation}
\label{eq:solSstar}
\alpha_{\ell_i,f(\ell_i)}=\bar\mu(\ell_i)+\displaystyle\sum_{k_j\in \mathscr D(\ell_i)} (-1)^{k-\ell}\bar\mu(k_j),\quad\ell\in\llbracket 1,d \rrbracket,\, i\in \llbracket 1,n_\ell \rrbracket,
\end{equation}
where sums over empty sets are understood as null. 
\end{lemma}

\begin{proof}
Regarding the leaves of $G$, it is immediate that we have 
\[\alpha_{d_i,f(d_i)}=\bar\mu(d_i),\quad i\in \llbracket 1,n_\ell \rrbracket.\]
The result then follows from a straightforward induction, by observing that \eqref{eq:Sstar} is equivalent for such $G$, to 
 \begin{equation*}
\alpha_{\ell_i,f(\ell_i)}=\bar\mu(\ell_i)-\displaystyle\sum_{(\ell+1)_j\in \mathscr S(\ell_i)} \alpha_{\ell_i,(\ell+1)_j},\quad\ell\in\llbracket 1,d \rrbracket,\, i\in \llbracket 1,n_\ell \rrbracket.
\end{equation*}
\end{proof}
\noindent From this, we first deduce the following result,  
\begin{proposition}
\label{prop:soltree}
Let $G=(\maV,\maE)$ be a tree having at least two nodes, and $\bar\mu\in\overline{\maM}(G)$.  
Then, the system (\ref{eq:S}) admits a unique solution in $\overline{\mathcal W}(G)$ if, and only if $\bar\mu(\maV_1)=\bar\mu(\maV_2)={1\over 2}$, where 
$\maV_1\cup \maV_2$ is the bipartition of $\maV$. Otherwise there is no solution to  (\ref{eq:S}). 
 \end{proposition}
 
 \begin{proof}
 Fix a root $r$, and let $\maV_1$ and $\maV_2$ be respectively given as the sets of nodes at even and odd generations, respectively. 
 For such $G$, the system (\ref{eq:S}) is equivalent to (\ref{eq:Sstar}) together with the relation 
 \[\bar\mu(r)=\sum\limits_{i=1}^{n_1} \alpha_{r,1_i}.\]
 By connectedness, all nodes of $G$ are descendants of the sons of $r$.  Therefore, from (\ref{eq:solSstar}) we have that 
 \begin{align}
 \bar\mu(r) &=\sum\limits_{i=1}^{n_1} \left(\bar\mu(1_i)+\displaystyle\sum_{k_j\in \mathscr D(1_i)} (-1)^{k-\ell}\bar\mu(k_j)\right)\notag\\
 	   &=\sum_{\substack{k=1;\\ k\tiny{\mbox{ odd}}}}^d \sum\limits_{i=1}^{n_k}\bar\mu(k_i) - \sum_{\substack{k=1;\\ k\tiny{\mbox{ even}}}}^d \sum\limits_{i=1}^{n_k}\bar\mu(k_i),\label{eq:M83}
 \end{align}
 which amounts to $\bar\mu(\maV_1)=\bar\mu(\maV_2)$. Therefore, the system \eqref{eq:S} admits the unique solution given by \eqref{eq:solSstar} if $\bar\mu(\maV_1)=\bar\mu(\maV_2)$, and no solution if not. 
 \end{proof}

\begin{proposition}
\label{prop:solcycle}
Let $G=(\maV,\maE)$ be a cycle, and $\bar\mu\in\overline{\maM}(G)$. Then,
\begin{enumerate}
\item If $G$ is of odd length, there exists a unique solution to \eqref{eq:S}. 
\item If $G$ is of even length, and we denote by $\maV_1\cup \maV_2$ the bipartition of $\maV$, then there are infinitely many solutions if 
$\bar\mu(\maV_1)=\bar\mu(\maV_2)={1\over 2}$; otherwise there is no solution. 
\end{enumerate}
\end{proposition}
\begin{proof}
Let $n$ be the size of the graph, and index the nodes of $G$ by $1,...,n$. The linear system \eqref{eq:S} of unknown 
$\vec{\alpha}=\begin{pmatrix}\alpha_{1,2}\,\alpha_{2,3}\,\alpha_{3,4}\, \cdots \,\alpha_{n-1,n}\,\alpha_{n,1}\end{pmatrix}^\prime$ 
can be written under the form 
\[A\vec \alpha = \vec \mu,\]
for $\vec{\mu}=\begin{pmatrix}\bar\mu(1)\,\bar\mu(2)\, \cdots \,\bar\mu(n)\end{pmatrix}^\prime$ and 
$A=\begin{pmatrix} 1 & 0 &\cdots  & & & 0 & 1\\
                                1 & 1 & 0 & \cdots & & & 0\\
                                0 & 1 & 1 & 0 & \cdots & &0\\
                                0 & 0 & 1 & 1 & 0 & \cdots & 0\\
                                &\\
                                \vdots  &  &  & \ddots &\ddots & \ddots &          \\
                                0 & 0 & \dots & & 0 & 1 & 1
                                \end{pmatrix}.$
                                
\noindent An immediate computation shows that $\det A = 1+(-1)^{n+1}$. 
Thus the system \eqref{eq:S} admits the unique solution $A^{-1}\vec \mu$ if $n$ is odd. 
If $n$ is even, then it is easily seen that $A$ has rank $n-1$, and that the system is compatible if and only 
if $$\sum_{\substack{i=1;\\ i\tiny{\mbox{ odd}}}}^n \bar\mu(i) =  \sum_{\substack{i=1;\\ i\tiny{\mbox{ even}}}}^n \bar\mu(i)={1\over 2},$$ and in that case there are 
infinitely many solutions.           
\end{proof}

\noindent We deduce the following result,
\begin{proposition}
\label{prop:soltreeplusone}
Let $G=(\maV,\maE)$ be a graph consisting of a tree having at least two nodes, and an additional edge, and let $\bar\mu\in\overline{\maM}(G)$.  
Then, 
\begin{enumerate}
\item If $G$ is a non-bipartite graph (i.e. the additional edge forms a cycle of odd length), then the system (\ref{eq:S}) admits a unique solution.
\item If $G$ is a bipartite graph (i.e. the additional edge forms a cycle of even length), the system (\ref{eq:S}) admits infinitely many solutions in $\overline{\mathcal W}(G)$ if, and only if $\bar\mu(\maV_1)=\bar\mu(\maV_2)={1\over 2}$. Otherwise, there is no solution to (\ref{eq:S}). 
\item If $G$ is the additional edge is a self-loop, the the system (\ref{eq:S}) admits a unique solution.
\end{enumerate} 
 \end{proposition}
\begin{proof}
Adding an edge that is no self-loop to the tree $G$, generates a cycle $\mathscr C$. If the cycle is of even size, then the resulting graph is bipartite, while if the cycle is of odd size we obtain a non-bipartite graph. 
\begin{enumerate}
\item First suppose that the resulting graph $G$ is non-bipartite, and is not reduced to an odd cycle (otherwise, assertion (1) of Proposition \ref{prop:solcycle} applies).  Then $G$ can be seen as the odd cycle $\mathscr C$, to which are appended one or several disconnected rooted trees. More precisely, denote by $N$ the (odd) number of nodes of $\mathscr C$, and by $r^1,...,r^p$ the nodes of $\mathscr C$ that are of degree more than 
then two, and by $s^1,...,s^{N-p}$ the nodes of $\mathscr C$ having degree 2, if any. Then, for any $l\in\llbracket 1,p \rrbracket$, to the node $r^l$ 
is appended a rooted tree $\mathscr T^l$ of root $r^l$. Lemma \ref{lemma:solrootedtree} yields that the weights associated to all edges of 
$\mathscr T^l$ are unique. 
In particular, denoting by $a^l$ and $b^l$ the two neighbors of $r^l$ in $\mathscr C$, 
\eqref{eq:solSstar} implies that 
\begin{equation}
\label{eq:stay}
\bar\mu(r^l)=\alpha_{r^l,a^l}+\alpha_{r^l,b^l}+\sum_{\substack{k=1;\\ k\tiny{\mbox{ odd}}}}^{d^l} \sum\limits_{i=1}^{n^l_k}\bar\mu(k^l_i) -\sum_{\substack{k=1;\\ k\tiny{\mbox{ even}}}}^{d^l} \sum\limits_{i=1}^{n^l_k}\bar\mu(k^l_i),
\end{equation}
using the same notations as in \eqref{eq:M83} for $\mathscr T^l$. 
This is true for any $l\in\llbracket 1,p \rrbracket$, so the restriction to $\mathscr C$ of any solution $\alpha$ to (\ref{eq:S}) solves in particular 
the system 
\begin{equation}
\label{eq:Sstarstar}
\tag{S$^{**}$}
\breve\mu(\ell)=\sum\limits_{j\in\mathcal{E}(\ell)\cap\mathscr C} \alpha_{l,j},\quad \ell\in \mathscr C,
\end{equation}
where 
\[\begin{cases}
\breve\mu(r^l) &=\bar\mu(r^l)-\sum\limits_{\substack{k=1;\\ k\tiny{\mbox{ odd}}}}^{d^l} \sum\limits_{i=1}^{n^l_k}\bar\mu(k^l_i) + \sum\limits_{\substack{k=1;\\ k\tiny{\mbox{ even}}}}^{d^l} \sum\limits_{i=1}^{n^l_k}\bar\mu(k^l_i),\quad l\in\llbracket 1,p \rrbracket,\\
\breve\mu(s^l) &=\bar\mu(s^l),\quad l\in\llbracket 1,N-p \rrbracket.
\end{cases}
\] 
From (1) of Proposition \ref{prop:solcycle}, after normalizing $\breve\mu$ the solution to \eqref{eq:Sstarstar} is unique, and so is the solution to \eqref{eq:S} in view of the uniqueness of the solution to (\ref{eq:Sstar}) on each tree $\mathscr T^l$, $l\in\llbracket 1,p \rrbracket$, from Lemma \ref{lemma:solrootedtree}. 
\item Suppose that the resulting graph $G$ is bipartite, and is not reduced to an even cycle (otherwise assertion (2) of Proposition 
\ref{prop:solcycle} applies). 
Applying the same construction as in case (1), we obtain again \eqref{eq:stay} for all $l$. 
Assertion (2) of Proposition \ref{prop:solcycle} shows that the resulting system 
\eqref{eq:Sstarstar} on $\mathscr C$ admits no solution unless $\breve\mu(\maV_1)=\breve\mu(\maV_2)$. But this is equivalent to 
$\bar\mu(\maV_1)=\bar\mu(\maV_2)$, as in each rooted tree $\mathscr T^l$ of $G$, the nodes of the even generations belong to the same subset of the bipartition as $r^l$. So under that condition, \eqref{eq:Sstarstar} and thereby 
\eqref{eq:S}, admit infinitely many solutions.
\item If the additional edge is a self-loop on node $r$, then $G$ can be seen as a rooted tree of root $r$. 
Then, it is immediate in view of Lemma \ref{lemma:solrootedtree} that \eqref{eq:S} admits a unique solution, that from \eqref{eq:solSstar}, 
is given by 
\begin{equation*}
\begin{cases}
\alpha_{\ell_i,f(\ell_i)} &=\bar\mu(\ell_i)+\displaystyle\sum_{k_j\in \mathscr D(\ell_i)} (-1)^{k-\ell}\bar\mu(k_j),\quad\ell\in\llbracket 1,d \rrbracket,\, i\in \llbracket 1,n_\ell \rrbracket,\\
\alpha_{r,r} &= \bar\mu(r)-\sum\limits_{\substack{k=1;\\ k\tiny{\mbox{ odd}}}}^{d} \sum\limits_{i=1}^{n_k}\bar\mu(k_i) + \sum\limits_{\substack{k=1;\\ k\tiny{\mbox{ even}}}}^{d} \sum\limits_{i=1}^{n_k}\bar\mu(k_i),
\end{cases}
\end{equation*}
which completes the proof.
\end{enumerate}
\end{proof}

Let us now come back to the representation of solutions of \eqref{eq:S} as asymptotic matching rates for corresponding matching models. 
\begin{definition}
Let $G=(\maV,\maE)$ be a multigraph (resp., a bipartite graph). We say that the matching rates are {\em policy-invariant} for $G$, if 
for any $\mu\in{\mathcal N}(\maV)$ (resp., any $\mu\in{\mathcal N}_{\tiny{\textsc{b}}}(\maV)$), the asymptotic matching rates $\Theta^{\Phi,\bar\mu}$ (resp., $\Theta^{\Phi,\tilde\mu}_{\scriptsize{\textsc{b}}}$) in the matching model $(G,\bar\mu,\Phi)$ (resp., $(G,\tilde\mu,\Phi)$) do not depend on $\Phi$ as long as $\bar\mu\in\textsc{Stab}(G,\Phi)$ (resp., $\tilde\mu\in\textsc{Stab}_{\tiny{\textsc{b}}}(G,\Phi)$). 
\end{definition}
Recalling equations \eqref{eq:defTheta} and \eqref{eq:defThetaBip}, and gathering Propositions \ref{prop:soltree}, \ref{prop:solcycle} and \ref{prop:soltreeplusone}, we readily obtain the following result, 

\begin{theorem}
Matching rates are policy-invariant for $G$ in the following cases:
\begin{itemize}
\item $G$ is an odd cycle,
\item $G$ is a tree,
\item $G$ a tree with an additional edge forming an odd cycle, 
\item $G$ is a tree with an additional self-loop. 
\end{itemize}
\end{theorem}

\subsection*{Acknowledgment} The authors would like to warmly thank Emmanuel Jeandel, for suggesting the use of Farkas Lemma to show the implication $(iv)\implies (ii)$ in Theorem \ref{thm:main}.

\end{document}